\documentclass[a4paper,11pt,reqno]{amsart}
\usepackage{lmodern}

\usepackage{amssymb}
\usepackage{siunitx}
\usepackage{epsfig}
\usepackage{amsfonts,amsrefs}
\usepackage{amsmath}
\usepackage{euscript}
\usepackage{amscd}
\usepackage{amsthm}
\usepackage{enumitem}
\DeclareMathAlphabet{\mathpzc}{OT1}{pzc}{m}{it}
\usepackage{enumitem}
\usepackage{color}
\usepackage{mathtools}
\usepackage[hypertexnames=false, colorlinks, citecolor=red, linkcolor=blue, urlcolor=red]{hyperref}
\usepackage[utf8]{inputenc}
\usepackage[T1]{fontenc} 
\usepackage{marginnote}
\usepackage{marvosym}

\usepackage[normalem]{ulem}

\newcommand{\blfootnote}[1]{%
  \begingroup
  \renewcommand\thefootnote{}\footnote{#1}%
  \addtocounter{footnote}{-1}%
  \endgroup
}

\newcommand{\marginextend}[1]{ \addtolength{\oddsidemargin}{-#1}  \addtolength{\evensidemargin}{-#1}
	\addtolength{\textwidth}{#1}\addtolength{\textwidth}{#1}}
\newcommand{\updownextend}[1]{ \addtolength{\topmargin}{-#1}  \addtolength{\textheight}{#1}
	\addtolength{\textheight}{#1}}
\marginextend{1.5cm}
\updownextend{0cm}
\allowdisplaybreaks[4]

\usepackage{pst-node}
\usepackage{tikz-cd}
\usepackage{mathrsfs}
\usepackage[most]{tcolorbox}

\DeclareFontFamily{OT1}{pzc}{}
\DeclareFontShape{OT1}{pzc}{m}{it}{<-> s * [1.10] pzcmi7t}{}
\DeclareMathAlphabet{\mathpzc}{OT1}{pzc}{m}{it}

\DeclareSymbolFont{SY}{U}{psy}{m}{n}
\DeclareMathSymbol{\emptyset}{\mathord}{SY}{'306}

\theoremstyle{plain}

\newtheorem{thm}{Theorem}[section]
\newtheorem*{thm*}{Theorem}
\newtheorem{cor}[thm]{Corollary}
\newtheorem{lem}[thm]{Lemma}
\newtheorem{prop}[thm]{Proposition}
\newtheorem{defn}[thm]{Definition}
\newtheorem{rem}[thm]{Remark}

\newtheoremstyle{mainthmstyle}%
  {}{}
  {\itshape}
  {}
  {\bfseries}
  {}
  {0pt}
  {\thmname{#1}. \thmnote{#3}} 

\theoremstyle{mainthmstyle}

\newtheoremstyle{named}{}{}{\itshape}{}{\bfseries}{.}{.5em}{#1 \thmnote{#3}}
\theoremstyle{named}

\tcolorboxenvironment{maintheorem}{
  colback=white,
  colframe=black,
  boxrule=0.8pt
}

\numberwithin{equation}{section}

\def\beq{\begin{eqnarray}}
	\def\eeq{\end{eqnarray}}
\def\beqa{\begin{eqnarray*}}
	\def\eeqa{\end{eqnarray*}}

\newcommand{\be}{\begin{equation}}
	\newcommand{\ee}{\end{equation}}
\newcommand{\bea}{\begin{eqnarray}}
	\newcommand{\eea}{\end{eqnarray}}
\newcommand{\Bea}{\begin{eqnarray*}}
	\newcommand{\Eea}{\end{eqnarray*}}

\newcommand{\clipnorm}{$_{\text{\rm c}}$Lip-norm}
\newcommand{\clipnorms}{$_{\text{\rm c}}$Lip-norms}

\newcounter{cnt1}
\newcounter{cnt2}
\newcounter{cnt3}
\newcommand{\blr}{\begin{list}{$($\roman{cnt1}$)$}
		{\usecounter{cnt1} \setlength{\topsep}{0pt}
			\setlength{\itemsep}{0pt}}}
	\newcommand{\bla}{\begin{list}{$($\alph{cnt2}$)$}
			{\usecounter{cnt2} \setlength{\topsep}{0pt}
				\setlength{\itemsep}{0pt}}}
		\newcommand{\bln}{\begin{list}{$($\arabic{cnt3}$)$}
				{\usecounter{cnt3} \setlength{\topsep}{0pt}
					\setlength{\itemsep}{0pt}}}
			\newcommand{\el}{\end{list}}
		
		\vspace{5mm}

\title[Metric dimension of $C^{\ast}$-algebras of cocycle twisted transformation groupoids]{Metric dimension of $C^{\ast}$-algebras of cocycle twisted transformation groupoids: Growth and dynamical complexity}

\author[A. Chattopadhyay]{Arnab Chattopadhyay}
\author[S. Joardar]{Soumalya Joardar}

\address[A. Chattopadhyay]{Indian Institute of Science Education And Research Kolkata, Mohanpur 741246, Nadia, West Bengal, India} \email{ac23rs002@iiserkol.ac.in}

\address[S. Joardar]{Indian Institute of Science Education And Research Kolkata, Mohanpur 741246, Nadia, West Bengal, India} \email{soumalya@iiserkol.ac.in}
	
\begin{document}
\begin{abstract}
 We consider a natural CQMS structure on a twisted transformation groupoid $C^{\ast}$-algebra coming from stratified \clipnorm\ as discussed in \cite{austad2026quantum}. We obtain upper bounds of metric dimension of reduced $C^{\ast}$-algebra of a transformation groupoid $\Gamma\rtimes X$ and its cocycle twist for a suitably chosen CQMS structure, provided $(X,d)$ is a compact metric space of finite Kolmogorov dimension and $\Gamma$ is a discrete group of polynomial growth. When $\Gamma$ has exponential growth, we prove that the dimension is generically $+\infty$ proving that the dichotomy between polynomial growth and exponential growth of groups survive even after considering cocycle twists of transformation groupoids. 
\end{abstract}

\maketitle
\section{Introduction and main results}
Compact quantum metric spaces are natural generalizations of classical compact metric spaces in the operator algebraic setting. It was initially proposed by A. Connes within his spectral triple framework (\cite{Connes1}). Then a general notion of compact quantum metric space was systematically developed by M. Rieffel in a series of papers (\cites{Rieffel2, Rieffel-1998-Metrics-on-state-from-action-of-cmpt-gp, Rieffel-2002-Group-C-alg-as-compact-quan-metr-sp, Rieffel1}). Since then, the literature on compact quantum metric spaces has received a wide contribution from many authors (\cites{Christensen, Long-Wu-2017-Twisted-group-C-alg-as-CQMS, latremoliere2013quantum, latremoliere2015quantum, latremoliere2016quantum, latremoliere2020convergence, latremoliere2022gromov, aguilar2023domains, MR2504529}). Since compact quantum metric spaces are generalizations of classical metric spaces, it is quite natural to extend quantities like classical Kolmogorov dimension, entropy of a classical compact metric spaces to the set up of compact quantum metric spaces. In \cite{Kerr}, the metric dimension was defined for a compact quantum metric spaces on unital $C^{\ast}$-algebra. It was shown that the metric dimension is able to recover the classical Kolmogorov dimension of compact metric spaces when applied to the unital $C^{\ast}$-algebra $C(X)$ in an appropriate framework. In the same paper Kerr also defined the product entropy of an automorphism of a unital $C^{\ast}$-algebra. The Kolmogorov dimension, various entropy are meaningful invariants which help in studying classical metric spaces, topological dynamical systems. So it is expected that Kerr's metric dimension and product entropy would play a similar role in the study of noncommutative metric geometry and noncommutative dynamical systems. But surprisingly, after the fundamental paper where Kerr computes metric dimension and product entropy for some concrete examples, there is hardly any follow-up. Motivated by this, the authors of this paper investigated metric dimension and product entropy in the context of natural compact quantum metric space structures on group $C^{\ast}$-algebras of finitely generated discrete groups in \cite{chattopadhyay2026metric}. It was found that the growth information of the underlying groups is largely remembered by the invariants like metric dimension and product entropy. \\
\indent In this context, it is quite natural to look at $C^{\ast}$-algebras of transformation groupoids as a testing ground for the invariants like metric dimension and entropy. When a discrete group $\Gamma$ acts on a compact metric space $X$, the corresponding transformation groupoid $\Gamma \rtimes X$ encodes both the topology of the phase space and the dynamics of the action. More generally, a twisting by a continuous $2$-cocycle produces a twisted transformation groupoid whose reduced $C^{\ast}$-algebra incorporates additional geometric and representation-theoretic information. These algebras form a natural class in which one may investigate the interaction between dynamics, operator algebras, and noncommutative metric geometry. The purely $C^{\ast}$-algebraic properties of transformation groupoid $C^{\ast}$-algebras are well studied and well understood. But the quantum metric space structures on the transformation groupoid $C^{\ast}$-algebras are largely unexplored. But to begin such investigation one has to start with some compact quantum metric structure. Thanks to a recent work of Austad (\cite{austad2026quantum}), there is a natural compact quantum metric space structure on general \'etale groupoid $C^{\ast}$-algebras with compact unit space. The structure is given by restricting the domain of the stratified \clipnorm\ to a sub-operator system consisting of functions with uniform finite support and subsequently exploiting rapid decay property with respect to some length function on the underlying groupoid to extend the CQMS structure to the domain of the stratified \clipnorm.
In particular, when a discrete group $\Gamma$ acts on a compact metric space $X$, the associated groupoid becomes an \'etale groupoid and one can adopt the similar techniques to obtain a class of compact quantum metric space structures on transformation groupoid $C^{\ast}$-algebras by invoking rapid decay property with respect to some length function on the underlying groupoid. In this paper, we also consider $2$-cocycle twists of transformation groupoid $C^{\ast}$-algebras and show that the natural compact quantum metric space structures can be extended to cocycle twists as well. We mainly focus on the understanding of metric dimensions of such structures and dependence of the metric dimension on the underlying phase space, the growth geometry of the acting discrete group and the dynamics in general. Now, let us briefly explain the main results obtained in this paper and their potential significance in the study of noncommutative geometry.\\
\indent The first important theorem in this paper is Theorem \ref{CQMS} where the main class of CQMS structures is obtained on the reduced $C^{\ast}$-algebra of a transformation groupoid and its $2$-cocycle twists. The main inputs required to obtain such CQMS structures are a length function on the transformation groupoid such that the groupoid has twisted rapid decay property in the sense of Definition \ref{rapid decay} and a stratification which is a collection of metrics on each fibre satisfying suitable hypothesis. A collection of metrics on each fibre is called a metric stratification. We consider two types of stratifications: static and dynamic. In static stratification, the metric on each fiber is the base metric of the compact phase space. We have our second main result of the paper (Theorem \ref{general cocycle bound} and Theorem \ref{betterbound}) which says that under the static stratification, the metric dimension of a cocycle twisted transformation groupoid $C^{\ast}$-algebra stays bounded provided the phase space has finite Kolmogorov dimension and the groupoid has polynomial growth property with respect to a length function that satisfies a suitable uniform Lipschitz condition \ref{Lip-length}. Then we feed the dynamics into the metric stratification and consider a dynamic stratification (see Subsection 3.1.2). Our choice makes the metric dimension read the dynamics. We have our third main result (Theorem \ref{dynamicbound}) which says that the metric dimension of the cocycle twisted groupoid $C^{\ast}$-algebra still remains finite for a finite dimensional phase space and polynomial growth provided the action is {\bf uniformly Lipschitz}. But we show by an example (Corollary \ref{dynamic_complexity}) that the metric dimension can blow up for a dynamic stratification even if the phase space has finite Kolmogorov dimension, one has polynomial growth and a uniformly equicontinuous action. The key takeaway is the following:\vspace{0.05in}\\
{\it The metric dimension can detect the dynamical complexity in the polynomial growth regime in an essential way, provided we feed the dynamics into the metric stratification to obtain the compact quantum metric space structure}.\vspace{0.05in}\\
\indent Our next important result is in the exponential growth regime. By the exponential growth regime, we mean the reduced $C^{\ast}$-algebra of a cocycle twisted transformation groupoid where the acting group admits exponential growth. This regime is already interesting in the context of rapid decay property. Unlike the case where the acting group has polynomial growth, in the exponential growth regime, the rapid decay property does not automatically pass from the group to the groupoid. We fix static stratification for exponential growth regime and investigate whether the growth complexity of the underlying group can influence the metric dimension. Our first main result (Theorem \ref{stratification_infinity}) in the exponential growth regime is that when the twisted transformation groupoid has the twisted rapid decay property, the metric dimension of the compact quantum metric space structure coming from static stratification is $+\infty$. We go on to prove a much stronger statement regarding the metric dimension. Indeed, we prove that (Theorem \ref{infinity_Mdim}) in the exponential growth regime, one can not have reasonably regular compact quantum metric space structure with finite metric dimension. In particular, it rules out any compact quantum metric space structure coming from a spectral triple in the sense of A. Connes. The results are able to recover the group $C^{\ast}$-algebra cases studied in \cite{chattopadhyay2026metric} by taking the phase space to be a singleton point and subsequently extend for cocycle twisted group $C^{\ast}$-algebras. This is intriguing even at the level of reduced untwisted group $C^{\ast}$-algebras from the perspective of noncommutative geometry. Let us explain why.\vspace{0.02in}\\
\indent On the one hand, it is known that every $K$-homology class of such algebras may admit representatives by finitely summable bounded Fredholm modules (\cite{emerson2018k}). On the other hand, for non-amenable groups, no finitely summable unbounded Fredholm module can exist (\cite{Connes1}), highlighting a striking disparity between bounded and unbounded representatives of the same $K$-homology classes. More generally, for exponential growth groups it remains unknown whether finitely summable unbounded Fredholm modules can occur in the amenable setting. Our results provide a complementary perspective on this question.  Since spectral triples constitute the unbounded representatives of $K$-homology classes and encode metric information absent from bounded Fredholm modules, our result suggests that exponential growth imposes an intrinsic obstruction at the level of noncommutative metric geometry. This phenomenon points to a fundamental distinction between homological finiteness and metric finiteness in noncommutative geometry. We believe that infinite metric dimension can be viewed as evidence that exponential growth forces any spectral-triple realization of the geometry to be ``infinite dimensional'' in a metric sense, regardless of the summability properties of bounded Fredholm representatives. So a key takeaway is the following rigidity phenomenon in the exponential growth regime:\vspace{0.05in}\\
{\it For the reduced $C^{\ast}$-algebra of a cocycle twisted transformation groupoid where the acting group has exponential growth, the metric dimension is generically $+\infty$ for a reasonable class of compact quantum metric space structures}.\vspace{0.05in}\\
 Summarizing, the results of this paper indicate that for a classical dynamical system, both the dynamical complexity of the dynamics and the growth complexity of the acting group are read in an essential way by compact quantum metric space structures on the reduced $C^{\ast}$-algebra of the cocycle twists of the corresponding transformation groupoids. The growth complexity is particularly striking. Because purely at the $C^{\ast}$-algebraic level, the growth information of the acting group might be lost. This is further reflected in the loss of rapid decay property in transformation groupoids coming from free actions of exponential growth group with rapid decay property (\cite{Weygandt-2024-Rapid-decay-for-etale-gpd}*{Theorem 3.2}) or \cite{buss2026rapid}*{Section 7.6}. In particular, the results of this paper emphasize the importance of metric dimension in the study of noncommutative geometry.\vspace{0.05in}\\
 {\bf Acknowledgement}: The first author acknowledges the financial support under the Senior Research Fellowship Scheme funded by UGC. The authors are grateful to Dr. Shubhabrata Das for several helpful discussions on geometric group theory.
\section{Preliminaries}
\subsection{Compact quantum metric space}
We begin this subsection with the definition of compact quantum metric space (CQMS for short). Compact quantum metric spaces can be defined for order unit spaces. But we shall deal with only unital $C^{\ast}$-algebras in this paper and therefore we define CQMS for unital $C^{\ast}$-algebras. We refer the reader to \cite{Rieffel1} for more details. Let $A$ be a unital $C^{\ast}$-algebra with unit $1_{A}$. A seminorm $L:A\rightarrow[0,+\infty]$ on $A$ possibly taking the value $+\infty$ is said to separate the state space $S(A)$ if given $\phi,\psi\in S(A)$ there is an element $a\in A$ such that $L(a)<+\infty$ and $\phi(a)\neq\psi(a)$. Given a separating seminorm $L$ on $A$, we denote the set $\{a\in A:L(a)<+\infty\}$ by $\mathrm{Dom}(L)$ and call it the domain of $L$. We denote the set $\{a\in\mathrm{Dom}(L):L(a)\leq 1\}$ by $L_1$. Then one defines a metric on the state space $S(A)$ by the following formula:
\begin{displaymath}
    d_{L}(\phi,\psi)=\sup\limits_{a\in L_{1}}\lvert\phi(a)-\psi(a)\rvert,
\end{displaymath}
for $\phi,\psi\in S(A)$. \\
A \clipnorm\ on $A$ is a separating seminorm $L:A\rightarrow[0,+\infty]$ such that
\begin{itemize}
    \item ${\textrm Ker}(L)=\mathbb{C}1_{A}$.
    \item $d_{L}$ metrizes the weak $\ast$-topology on $S(A)$.
\end{itemize}
\begin{defn}\label{CQMSdefn}
    A compact quantum metric space is a pair $(A,L)$ such that $A$ is a unital $C^{\ast}$-algebra and $L$ is a \clipnorm\ on $A$.
\end{defn}
A CQMS structure on a unital $C^{\ast}$-algebra $A$ is aid to come from a spectral triple if\\
(i) $\mathcal{A}\subset A$ is some dense $\ast$-subalgebra and $(\mathcal{A},H,D)$ is some spectral triple.\\
(ii) The densely defined seminorm $L$ on $\mathcal{A}$ given by $a\mapsto\lvert\lvert[D,a]\rvert\rvert$ is a \clipnorm\ in our sense.\\
It is easy to see that in this case the \clipnorm\ $L$ satisfies the Leibnitz property $L(ab)\leq L(a)\lvert\lvert b\rvert\rvert+\lvert\lvert a\rvert\rvert L(b)$.
\begin{rem}
    We have essentially taken the definition of a \clipnorm\ as given in Definition 2.3 of \cite{Kerr}. But we have dropped the adjoint invariance of the \clipnorm\ . Note that under the assumption of adjoint invariance, for a CQMS structure, it is enough to look at the self-adjoint part of the unit Lip-ball to recover the metric on the state space (see the discussion after \cite{Kerr}*{Definition 2.1}). We do not require that and we shall see that this asymmetry is forced in our main example of CQMS. Neither do we assume any Leibnitz property of the $_{\text{\rm c}}$Lip-norm. M. Rieffel originally considered Lip-norms for his definition of compact quantum metric spaces, but we stick to the \clipnorms\ in the sense of \cite{Kerr} without the adjoint invariance condition. In our case as $\mathrm{Dom}(L)$ will always be dense in $A$, $L$ will separate points in $S(A)$ in the sense of \cite{Kerr}. 
\end{rem}
Now we recall the definition of metric dimension of a CQMS $(A,L)$ where $A$ is a unital $C^{\ast}$-algbera. To that end we recall a few notations from \cite{Kerr}. For a normed linear space $(X,\| \cdot \|)$ (either a $C^{\ast}$-algebra or a Hilbert space for us), $\mathcal{F}(X)$ will denote the collection of all finite dimensional subspaces of $X$. If $Y,Z$ are subsets of $X$, then for $\delta>0$, the notation $Y\subseteq_{\delta} Z$ will mean that for every $y\in Y$, there is some $z\in Z$ such that $\lvert\lvert y-z\rvert\rvert<\delta$. For any subset $Y\subset X$, $D(Y,\delta)=\inf\{{\textrm{dim}}(Z):Z\in\mathcal{F}(X), Y\subseteq_{\delta}Z\}$, where $\textrm{dim}(Z)$ is the vector space dimension of $Z$. Then it is easy to see that if $Y_{1}\subseteq Y_{2}$, $D(Y_{1},\delta)\leq D(Y_{2},\delta)$ for any $\delta>0$. Now let $(A,L)$ be a compact quantum metric space on a unital $C^{\ast}$-algebra in the sense of Definition \ref{CQMSdefn}. Then we denote the set $\{a\in A:L(a)\leq 1, \lvert\lvert a\rvert\rvert\leq 1\}$ by $\mathcal{L}_1$. Then we have the following:
\begin{defn}
    The metric dimension of a CQMS $(A,L)$ is defined to be
    \begin{displaymath}
    \mathrm{Mdim}_{L}(A):=\limsup\limits_{\delta\to 0^{+}} \frac{\log D(\mathcal{L}_{1},\delta)}{\log  \delta^{-1}}
    \end{displaymath}
\end{defn}
\begin{rem}
     Notice that The set $\mathcal{L}_{1}$ chosen here is slightly different from the set $\mathcal{L}_1$ chosen in \cite{Kerr} to define the metric dimension. But our choice is able to recover all the existing results about metric dimension. 
\end{rem}
It turns out that for any compact metric space $(X, d),$ the metric dimension of of $C (X)$ with respect to the associated \clipnorm\ on $C (X)$ coincides with its Kolmogorov dimension. To this end, we shall recall Kolmogorov dimension of a compact metric space $(X, d).$ For more details, the readers can refer to \cites{MR124720, MR1187754}.
\begin{defn}
Given a compact metric space $(X, d),$ let $N (\delta, d)$ be the minimal cardinality of a cover of $X$ by $\delta$-balls. Then the \textit {Kolmogorov dimension} of $(X, d),$ denoted by $\dim_B (X, d),$ is defined as 
$$\dim_B (X, d) = \limsup\limits_{\delta \to 0^{+}} \frac {\log N (\delta, d)} {\log \delta^{-1}}.$$
Kolmogorov dimension is often alternatively called the \textit {upper box counting dimension} or \textit {upper Minkowski dimension}.
\end{defn}
It is worth noting that Kerr’s proof of the metric dimension as an upper bound of the Kolmogorov dimension relies exclusively on elements of norm $1.$ Therefore the following proposition is a verbatim adaptation of \cite{Kerr}*{Proposition 3.9}. 
\begin{prop}
Let $(X, d)$ be a compact metric space and $L_X$ be the associated \clipnorm\ on $C (X),$ that is, $$L_X (f) = \sup\limits_{x \neq y} \frac {\left \lvert f (x) - f (y) \right \rvert} {d (x, y)},$$ for all $f \in C (X).$ Then
$$\mathrm {Mdim}_{L_X} (C (X)) = \dim_B (X, d).$$
\end{prop}
Now we recall the bi-Lipschitz equivalence which is the relevant equivalence in the context of CQMS theory.
\begin{defn}
    \label{bi-lipschitz_equivalence}(\cite{Kerr}*{Definition 2.8}) Let $A,B$ be two unital $C^{\ast}$-algebras with \clipnorm\ $L_{A}$, $L_{B}$ respectively. A positive unital linear map $\phi:A\rightarrow B$ is said to be Lipschitz if there is a $\lambda\geq 0$, such that $L_{B}(\phi(a))\leq \lambda L_{A}(a)$ for all $a\in \mathrm{Dom}(L_{A})$. If $\phi$ is invertible and both $\phi,\phi^{-1}$ are Lipschitz then we say $\phi$ is bi-Lipschitz and in that case we say $(A,L_A)$ and $(B,L_B)$ are bi-Lipschitz equivalent.
\end{defn}
The metric dimension is an invariant for bi-Lipschitz equivalence. We recall the following theorem whose proof would be a slight modification of the existing proof because of our different choice of $\mathcal{L}_1$. It follows from the observation that positive unutal maps are $C^{\ast}$-norm contractive.
\begin{thm}
    \label{Metric_inv}(\cite{Kerr}*{Proposition 3.4}) Let $(A,L_A)$ and $(B,L_{B})$ be two bi-Lipschitz equivalent CQMS. Then 
    \begin{displaymath}
        \mathrm{Mdim}_{L_A}(A)=\mathrm{Mdim}_{L_{B}}(B).
    \end{displaymath}
\end{thm}
\subsection{Transformation groupoid \texorpdfstring{$C^{\ast}$}{C*-}-algebras and their cocycle twists} 
Let $\Gamma$ be a discrete group acting on a compact Hausdorff topological space $X$ by homeomorphisms. Consider the groupoid $\mathcal G : = \Gamma \rtimes X$ with the range map $r$ and the source map $s$ given by $r(t,x) = t \cdot x$ and $s(t,x) = x$ respectively. Then $(t,y), (u,x)$ is a composable pair if and only if $y=u\cdot x$ and in that case the  multiplication is defined by $(t, u \cdot x) (u, x) = ( t u, x)$. The inverse is given by $(t, x)^{-1} = (t^{-1}, t \cdot x)$. With this structure, $\mathcal G$ becomes an étale groupoid, known as the transformation groupoid, with the unit space $\mathcal G^{(0)}=\{e\} \times X : = \{(e,x)\ :\ x \in X\}$ canonically identified with $X$. In this paper, we shall not recall general étale groupoids. We shall stick to the transformation groupoids only. Most of the definitions that follow are adapted from more general étale groupoids to transformation groupoids. In this paper, the group $\Gamma$ in the transformation groupoid $\mathcal{G}$ is assumed to be finitely-generated discrete.
\vspace{2mm}

\begin{defn}(\cite{renault1980groupoid}*{Definition 1.12}, \cite{ARMSTRONG2022109551}*{Definition 2.2})
A $2$-cocycle $\omega$ on a transformation groupoid $\Gamma \rtimes X$ is map $\omega : \Gamma \times \Gamma \times X \rightarrow \mathbb T$ which satisfies the following cocycle condition \begin{displaymath}
\omega (t, u, v \cdot x) \omega (t u, v, x) = \omega (t, u v, x) \omega (u, v, x).\end{displaymath}
A $2$-cocycle $\omega$ is said to be normalized if $\omega (e, t, x) = \omega (t, e, x) = 1,$ for all $t \in \Gamma$ and $x \in X.$
\end{defn}
\begin{rem}
    We shall only consider normalized $2$-cocycles. So from now on a $2$-cocycle will always be normalized unless mentioned otherwise.
\end{rem}
Recall that $2$-cocycle twists on an étale groupoid are in one-one correspondence with twists with respect to the trivial $\mathbb{T}$-bundles in the sense of \cite{Weygandt-2024-Rapid-decay-for-etale-gpd}*{Definition 2.3, Example 2.4}. For a proof of this equivalence, the reader is referred to \cite{ARMSTRONG2022109551}*{Proposition 2.9}. In this paper we shall not consider general twist of a groupoid, rather we shall stick to only $2$-cocycle twist.\vspace{0.1in}\\
Given a transformation groupoid one can associate two $C^{\ast}$-algebras, the full and the reduced $C^{\ast}$-algebras. In this paper, we shall be considering the reduced $C^{\ast}$-algebras. We briefly recall its construction. The reader may refer to \cites{williams2019tool, renault1980groupoid, brownc, Weygandt-2024-Rapid-decay-for-etale-gpd, sims2020operator} for detailed treatment of reduced $C^{\ast}$-algebras of general \'etale groupoids. Given a transformation groupoid equipped with a normalized $2$-cocycle $\omega,$ we can associate its twisted reduced $C^{\ast}$-algebra. For $\mathcal G = \Gamma \rtimes X,$ let $C_c (\mathcal G, \omega)$ be the $\ast$-algebra of compactly supported continuous functions on $\mathcal G$ with the multiplication given by 
\begin{displaymath}(f \ast_{\omega} g) (t, x) = \sum\limits_{u \in \Gamma} f \left (t u^{-1}, u \cdot x \right ) g (u, x) \omega \left (t u^{-1}, u, x \right )\end{displaymath}
and the involution given by
$f^{\ast_{\omega}} (t, x) = \overline {f \left (t^{-1}, t \cdot x \right ) \cdot \omega (t, t^{-1}, t \cdot x)},$
 for $f, g \in C_c (\mathcal G, \omega)$ and $(t, x) \in \mathcal G.$
Consider the function $\delta_e \colon  \mathcal G \to \mathbb C$ defined by 
\begin{equation}\label{eq:0} \delta_e (t, x) = \begin{cases} 1, \quad \text {if}\ t = e, x \in X, \\ 0, \quad \text {otherwise.} \end{cases}\end{equation}
Note that the unit space $X$ of the transformation groupoid $\mathcal G = \Gamma \rtimes X$ is always clopen. Also since $X$ is compact, it follows that $\delta_e \in C_c (\mathcal G, \omega)$. It is straightforward to verify that $f \ast_{\omega} \delta_e = \delta_e \ast_{\omega} f$ for every $f \in C_c (\mathcal G,\omega)$, that is, $\delta_e$ is the multiplicative identity of $C_c (\mathcal G,\omega).$ One defines a $C(X)$-valued inner product $\langle \langle \cdot, \cdot \rangle \rangle$ on $C_c (\mathcal G, \omega)$ by
\begin{displaymath}\langle \langle \xi, \eta \rangle \rangle (x) = \sum\limits_{t \in \Gamma} \overline {\xi (t, x)}\ {\eta (t, x)}\end{displaymath} 
for $\xi, \eta \in C_c (\mathcal G, \omega)$ and $x \in X.$ Let $L^2 (\mathcal G, \omega)$ be the Hilbert-$C (X)$-module obtained by the completion of $C_c (\mathcal G, \omega)$ with respect the norm $\| \cdot \|$ given by
\begin{displaymath}
    \|\xi\| : = \biggr ( \sup\limits_{x \in X}  \sum\limits_{t \in \Gamma} \left \lvert \xi(t, x) \right \rvert^2 \biggr )^{\frac {1} {2}},\end{displaymath}
for $\xi \in C_c (\mathcal G, \omega).$ Consider the representation $\lambda^{\omega} \colon C_c (\mathcal G, \omega) \to \mathbb B (L^2 (\mathcal G, \omega))$ of $C_c (\mathcal G, \omega)$ into the space of adjointable operators on $L^2 (\mathcal G, \omega)$ given by $\lambda^{\omega} (f) (\xi) = f \ast_{\omega} \xi,$ for $\xi \in C_c (\mathcal G, \omega).$ Recall the $I$-norm for an element $f\in C_{c}(\mathcal{G},\omega)$: 
$$\|f\|_I : = \max \biggr \{\sup\limits_{x \in X} \sum\limits_{t \in \Gamma} \left \lvert f (t, x) \right \rvert,\ \sup\limits_{x \in X} \sum\limits_{t \in \Gamma} \left \lvert f \left (t^{-1}, t \cdot x \right ) \right \rvert \biggr \}.$$ 
We define the twisted reduced norm $\|\cdot \|_{\text {red}, \omega}$ on $C_c (\mathcal G,\omega)$ by
\begin{displaymath}\|f\|_{\text {red}, \omega} : = \|\lambda^{\omega} (f) \|_{\text {adj}}.\end{displaymath} 
The completion of $C_c (\mathcal G, \omega)$ with respect to the twisted reduced norm is known as the twisted reduced groupoid $C^{\ast}$-algebra, which is denoted by $C_r^{\ast} (\mathcal G, \omega).$ It is easy to see that $C_c (\mathcal G, \omega)$ can be identified with its image in $L^2 (\mathcal G, \omega).$\\
Now we note an inequality which will be used throughout the paper. The proof of the inequality follows from easy adaptation of the proof of the same inequality for the untwisted case found in \cite{brownc}*{Lemma 5.6.12}. Recall the fiberwise twisted representation $\{\lambda^{\omega}_{x}\}_{x\in X}$ of $C_{c}(\mathcal{G},\omega)$ in $\mathbb{B}(\ell^{2}(\Gamma))$ given by 
\begin{displaymath}
    \lambda^{\omega}_x (f) (\eta) (t) = \sum\limits_{u \in \Gamma} f \left (t u^{-1}, u \cdot x \right ) \eta (u)\ \omega \left (t u^{-1}, u, x \right ),
\end{displaymath}
 for all $f \in C_c (\mathcal G, \omega), \eta \in \ell^{2} (\Gamma)$ and $t \in \Gamma.$ Then denoting $$\|f\|_{\infty} : = \sup\limits_{\substack {t \in \Gamma \\ x \in X}} \left \lvert f (t, x) \right \rvert,$$ for all $f \in C_c (\mathcal G, \omega),$ we have
 \begin{equation}\label{I-n0rm dominance}
     \lvert\lvert f\rvert\rvert_{\infty}\leq \lvert\lvert f\rvert\rvert_{\omega,\mathrm{red}}=\sup\limits_{x\in X}\lvert\lvert\lambda^{\omega}_{x}(f)\rvert\rvert\leq\lvert\lvert f\rvert\rvert_{I}.
 \end{equation}
Before we proceed, we recall the canonical state $\tau$ on $C^{\ast}_{r}(\mathcal{G},\omega)$ given by $\mu\circ\mathbb{E}$, where $\mu$ is some state on $C(X)$ and $\mathbb{E}:C^{\ast}_{r}(\mathcal{G},\omega)\to C(X)$ given on $C_{c}(\mathcal{G},\omega)$ by $\mathbb{E}(\sum\limits_{t}\delta_{t}a_{t})=a_{e}$. The proof of existence of a faithful conditional expectation $\mathbb{E}$ for reduced $C^{\ast}$-algebra of a general twisted Hausdorff locally compact étale groupoid can be found in \cite{renault2008cartan}*{Proposition 4.3}. It is easy to translate the result to our setting. Then the corresponding GNS Hilbert space has $[\delta_{e}1]$ as the cyclic vector. We denote the GNS Hilbert space by $L^{2}(\tau)$ and the representation on the GNS space by $\Lambda_{\tau}$.
\begin{lem}
\label{Voiculescu_orthonormal}
Let $\mathcal G = \Gamma \rtimes X$ be a transformation groupoid obtained from an action of a discrete group $\Gamma$ on a compact Hausdorff space $X$ equipped with a probability measure $\mu.$ Let $F$ be a finite set of $\mu$-orthonormal functions in $C (X).$ Then for any finite subset $\Gamma' \subseteq \Gamma,$ the set $$\Omega :  = \{\Lambda_{\tau} (\delta_t a)[\delta_e 1]\ :\ t \in \Gamma',\ a \in F \} \subseteq L^2(\tau)$$ is a finite set of orthonormal vectors in $L^{2} (\tau).$ In particular, for any finite subset $\Gamma^{\prime}\subset\Gamma$, the set $\{\Lambda_{\tau} (\delta_t 1)[\delta_e 1]\ :\ t \in \Gamma^{\prime}\}$ is an orthonormal set in $L^{2}(\tau)$ for any choice of probability measure $\mu$ on $X$. 
\end{lem}

\begin{proof}
For any two pairs $\left (t_1, a_1 \right )$ and $\left (t_2, a_2 \right )$ in $\Gamma' \times F,$ we have
\Bea
\langle \Lambda_\tau(\delta_{t_1} a_{1})[\delta_e 1], \Lambda_\tau(\delta_{t_2} a_{2})[\delta_e 1] \rangle_{2, \tau} 
& = & \tau \left( (\delta_{t_1} a_{1})^{\ast_{\omega}} \ast_{\omega} \left (\delta_{t_2} a_{2} \right ) \right) \\
& = & \mu \circ \mathbb{E} \left(\delta_{t_1^{-1}} \left (\alpha_{t_1} \left (\overline {a_{1}} \right ) \overline {\omega \left (t_1, t_1^{-1}, \cdot \right )} \right ) \ast_{\omega} \left (\delta_{t_2} a_{2} \right ) \right ) \\ & = & \mu \circ \mathbb {E} \left (\delta_{t_1^{-1} t_2} \left( \alpha_{t_2^{-1} t_1}(\overline{a_{1}})\ a_{2}\ \overline{\omega(t_1, t_1^{-1} t_2, \cdot)} \right) \right )
\Eea
The conditional expectation vanishes identically when $t_1 \neq t_2$. If $t_1 = t_2$, the expression reduces to $\mu(\overline {a_{1}} a_{2})$. By the $\mu$-orthonormality of the elements of $F,$ this evaluates to $0$ if $a_1 \neq a_2$ and $1$ if $a_1 = a_2.$ This completes the proof.
\end{proof}

\subsection{Polynomial growth and twisted rapid decay property} We shall recall the rapid decay property of twisted reduced transformation groupoid $C^{\ast}$-algebras. Again this will be adaptation of the same for more general \'etale groupoid case and the reader is referred to \cite{Weygandt-2024-Rapid-decay-for-etale-gpd} for details.

\begin{defn}
Let $\mathcal G : = \Gamma \rtimes X$ be a transformation groupoid associated to an action of a discrete group $\Gamma$ on a compact Hausdorff space $X.$ By a length function on $\mathcal G,$ we mean a map $\ell : \mathcal G : \longrightarrow [0, +\infty)$ satisfying the following conditions $:$

\vspace{2mm}

$(1)$ $\ell (e, x) = 0$ for any $x \in X,$ i.e. $\ell$ vanishes exactly on the unit space of $\mathcal G,$

$(2)$ $\ell \left (t^{-1}, t \cdot x \right ) = \ell (t, x)$ for any $(t,x) \in \mathcal G,$ i.e. $\ell$ is inverse invariant, and,

$(3)$ $\ell (t u, u \cdot x) \leq \ell (t, u \cdot x) + \ell (u, x)$ for any $t, u \in \Gamma$ and $x \in X,$ i.e. $\ell$ is finitely subadditive on the space of composable pairs of $\mathcal G.$

\vspace{2mm}

We say that $\ell$ is continuous if it is continuous as a map from $\mathcal G$ to $[0, \infty).$ The length function \(\ell\) is called proper if for every subset $K \subseteq \mathcal G \setminus \left (\{e\} \times X \right ),$ finiteness of the quantity $\sup \{\ell (t, x) : (t, x) \in K\}$ implies that $K$ is pre-compact.

\end{defn}

\begin{rem}
We shall be only considering continuous length functions in this paper. Henceforth all the length functions that will be dealt with is assumed to be continuous unless specified explicitly.
\end{rem}

\begin{defn}
A transformation groupoid $\mathcal G : = \Gamma \rtimes X$ equipped with a length function $\ell$ is said to have the property of polynomial growth with growth exponent $r \geq 1$ with respect to $\ell$ if there exists a constant $C > 0$ such that for each $n \geq 0,$ we have 
$$\sup\limits_{x \in X} \left \lvert B_x (n) \right \rvert \leq C (1 + n)^{r},$$
where $B_x (n) : = \left \{t \in \Gamma\ :\ \ell (t, x) \leq n \right \}.$
\end{defn}
\begin{rem}
Many authors prefer to involve both the source ball $B_{\mathcal G_u} (n)$ and the range ball $B_{\mathcal G^{u}} (n)$ (for $u \in \mathcal G^{(0)}$) while defining the property of polynomial growth for a general \'etale groupoid as in \cite{austad2024polynomial}*{Definition 3.11}, however there is a one to one correspondence between $B_{\mathcal G_u} (n)$ and $B_{\mathcal G^{u}} (n)$ for any $n \geq 0,$ given by $x \mapsto x^{-1},$ as the length function is inverse invariant. Note that the polynomial growth of any group length function gives rise to the polynomial growth of the base independent induced length function on the transformation groupoid with the same growth exponent.
\end{rem}

\begin{lem} \label{induced length}
Let $\ell$ be a length function on the transformation groupoid $\mathcal{G} := \Gamma \rtimes X$, where $\Gamma$ is a discrete group acting on a compact Hausdorff space $X.$ Define a map $\ell_{\Gamma} \colon \Gamma \longrightarrow [0, \infty)$ by 
\[ \ell_{\Gamma}(t) := \sup_{x \in X} \ell(t, x). \]
Then $\ell_{\Gamma}$ is a length function on $\Gamma$. Moreover, if $\ell$ is proper, then so is $\ell_{\Gamma}$ and if $\mathcal{G}$ has polynomial growth with respect to $\ell$, $\Gamma$ has polynomial growth with respect to $\ell_{\Gamma}$ with the same exponent.
\end{lem}

\begin{proof}
The fact that $\ell_{\Gamma}$ is a length function on $\Gamma$ is straightforward and the proof of this fact is left to the reader.

\vspace{0.5mm}

We prove that if $\ell$ is proper then so is $\ell_{\Gamma}.$ Since $\ell$ is a proper length function on the groupoid $\Gamma \times X$, for any $R > 0$, the set 
$$K_R = \{ (t, x) \in \Gamma \times X \mid \ell(t, x) \le R \}$$
is pre-compact in the product topology which implies that its closure $\overline{K_R}$ is compact. We need to show that $\ell_{\Gamma}$ is proper on $\Gamma.$ In other words we need to show that the set 
\[ S_R = \{ t \in \Gamma \mid \ell_{\Gamma}(t) \le R \} \]
is finite for every $R > 0.$

\vspace{2mm}

Let $p_1 \colon \Gamma \times X \to \Gamma$ be the continuous projection map onto the first factor. Since continuous maps preserve compactness, $p_1(\overline{K_R})$ is a compact subset of $\Gamma$. Since $\Gamma$ is discrete, the compact set $p_1(\overline{K_R})$ has to be finite. But then $p_1(K_R)$ is finite as well.
Let $t \in S_R.$ By definition, $\ell_{\Gamma}(t) \le R,$ which implies that $$\sup_{x \in X} \ell(t, x) \le R.$$ Therefore, $\ell(t, x) \le R$ for all $x \in X$. Fixing an arbitrary $x_0 \in X$, we have $\ell(t, x_0) \le R.$ This ensures that $(t, x_0) \in K_R.$
Applying the projection map we obtain
$$t = p_1(t, x_0) \in p_1(K_R).$$
This shows that $S_R \subseteq p_1(K_R).$ Since $p_1(K_R)$ is a finite set, $S_R$ must also be finite. This shows that $\ell_{\Gamma}$ is a proper length function on $\Gamma,$ as claimed. The claim about the polynomial growth property essentially follows from the definition.
\end{proof}
\begin{rem}
    The last lemma shows that if $\mathcal{G}=\Gamma\rtimes X$ has polynomial growth property with respect to some length function $\ell$, then $\Gamma$ has polynomial growth property with respect to the length function $\ell_{\Gamma}$. Therefore, in that case, the group $\Gamma$ cannot have exponential growth with respect to any word length function. Unlike the property of polynomial growth, the property of rapid deacy for the group length function does not necessarily lift to the property of rapid decay of the fiber independent induced length function on the transformation groupoid if the underlying group has exponential growth. For further details, readers may refer to \cite{Weygandt-2024-Rapid-decay-for-etale-gpd}*{Theorem 3.2}.
\end{rem}
\begin{lem} \label{length-Lip}
Let $\Gamma$ be a discrete group acting on a compact metric space $(X, d)$ with finite diameter $D = \sup\limits_{x,y \in X} d(x,y).$ Let $\Gamma \rtimes X$ be the associated transformation groupoid. Suppose $\ell : \Gamma \rtimes X \to \mathbb{R}_{\ge 0}$ is a length function satisfying the condition
\begin{equation} \label{Lip-length}
    |\ell(t, x) - \ell(t, y)| \le C d(x, y)
\end{equation}
for some constant $C > 0$ (independent of $t \in \Gamma$) and all $t \in \Gamma,$ $x, y \in X$. Recall the length function $\ell_{\Gamma}$ on $\Gamma$ by $\ell_{\Gamma}(t) : = \sup\limits_{x \in X} \ell(t, x)$. Fix $n_0 \in \mathbb{N}$. If $m_0$ is the smallest positive integer such that the implication $\ell_{\Gamma}(t) > m_0 \implies \ell(t, x) > n_0$ holds for all $x \in X$, then $m_0 \le \lceil n_0 + CD \rceil$.
\end{lem}

\begin{proof}
We proceed by contraposition. Consider the subset $F_{n_0}$ in $\Gamma$ defined by
\begin{equation*}
    F_{n_0} : = \{ t \in \Gamma\ :\ \exists x \in X \text{ such that } \ell(t, x) \le n_0 \}.
\end{equation*}
By the definition of $m_0,$ it follows that $m_0$ is the least positive integer such that $\ell_{\Gamma}(t) \le m_0$ for all $t \in F_{n_0}.$

\vspace{2mm}

Fix an arbitrary $t \in F_{n_0}$. By the definition of the set $F_{n_0}$, there exists a point $x_t \in X$ satisfying $\ell(t, x_t) \le n_0.$ Thus by invoking the given condition we have
\begin{equation*}
    \ell(t, y) - \ell(t, x_t) \le |\ell(t, y) - \ell(t, x_t)| \le C d(y, x_t),
\end{equation*}
for all $y \in X.$
Since the metric space $X$ has diameter $D$, we have $d(y, x_t) \le D.$ In other words,
\begin{equation*}
    \ell(t, y) \leq \ell(t, x_t) + CD \le n_0 + CD.
\end{equation*}
As this inequality holds for all $y \in X,$ we must have
\begin{equation*}
    \ell_{\Gamma}(t) = \sup_{y \in X} \ell(t, y) \le n_0 + CD.
\end{equation*}
Since this upper bound is uniform over all $t \in F_{n_0},$ the supremum of $\ell_{\Gamma}(t)$ over $F_{n_0}$ is bounded by $n_0 + CD.$ Therefore, by the minimality of $m_0$ it follows that
\begin{equation*}
    m_0 \leq \left \lceil n_0 + CD \right \rceil.
\end{equation*}
This completes the proof.
\end{proof}
\begin{rem}\label{grouplengthtogroupoidlength}
    An example of a length function on $\Gamma\rtimes X$ that satisfies Condition \ref{Lip-length} is given by $\ell(t,x):=\ell_{\Gamma}(t)$ for all $(t,x)\in\Gamma\times X$ where $\ell_{\Gamma}$ is some length function on the group $\Gamma$. Also it is easy to see that if $\Gamma$ has polynomial growth with respect to $\ell_{\Gamma}$, then $\Gamma\rtimes X$ has polynomial growth with respect to $\ell$ with the same growth exponent.
\end{rem}
Now we shall define the property of rapid decay for reduced $C^{\ast}$-algebras of transformation groupoids with respect to a length function. For general groups, the study of this property was initiated by P. Jolissaint (\cite{MR943303}) and subsequently developed by I. Chatterjee (\cites{ChatterjiRuane2005, Chatterji2016}). The rapid decay property for reduced $C^{\ast}$-algebras of general \'etale groupoids was generalized by Alex Weygandt in the twisted setting (\cite{Weygandt-2024-Rapid-decay-for-etale-gpd}) and subsequently extended to the framework of Fell bundles over \'etale groupoids by A. Buss and P. Karmakar (\cite{buss2026rapid}).
\begin{defn}
\label{rapid decay}
The transformation groupoid $\mathcal G : = \Gamma \rtimes X$ is said to have the property of twisted rapid decay with decay exponent $p > 0$ with respect to $\ell$ and the normalized $2$-cocycle $\omega : \Gamma \times \Gamma \times X \rightarrow \mathbb T$ if there exists a constant $C > 0$ such that for all $f \in C_c (\mathcal G, \omega),$ we have
$$\|f\|_{\mathrm {red}, \omega} \leq C \|f\|_{2, p, \ell},$$
where
$$\|f\|_{2, p, \ell} : = \max \left \{\|f\|_{2, p, s, \ell}, \|f\|_{2, p, r, \ell} \right \},$$ with
$$\|f\|_{2,p,s,\ell} : = \sup\limits_{x \in X} \left (\sum\limits_{t \in \Gamma} \left \lvert f(t, x) \right \rvert^2 (1 + \ell (t, x))^{2 p} \right )^{\frac {1} {2}},$$ 
and,
$$\|f\|_{2,p,r, \ell} : = \sup\limits_{x \in X} \left (\sum\limits_{t \in \Gamma} \left \lvert f \left (t^{-1}, t \cdot x \right ) \right \rvert^2 (1 + \ell (t, x))^{2 p} \right )^{\frac {1} {2}}.$$
\end{defn}
 
\begin{lem}
\label{polytorapid}
Let $\mathcal G : = \Gamma \rtimes X$ be the transformation groupoid associated to an action of discrete group $\Gamma$ on a compact Hausdorff space $X,$ equipped with a $2$-cocycle $\omega : \Gamma \times \Gamma \times X \rightarrow \mathbb T.$ Let $\ell : \mathcal G \rightarrow [0, \infty)$ be a continuous proper length function on $\mathcal G$ with respect to which $\mathcal G$ has the property of polynomial growth with growth exponent $r \geq 1.$ Then $\mathcal G$ has the property of twisted rapid decay with respect to $\ell$ with decay exponent $p > \frac {r} {2}.$
\end{lem}
There is already a version of the above lemma for finitely generated untwisted discrete groups in \cite{nica2010degree}*{Proposition 2.2}. The main idea of this proof is to exploit the fact that the groups of polynomial growth are amenable and that any amenable group satisfies rapid decay like property if and only if the growth $\gamma$ of the group $\Gamma$ satisfies $\gamma (n) \leq C n^{2 s},$ for some constant $C > 0$ (independent of $n \in \mathbb N$). However, for the transformation groupoid $C^{\ast}$-algebra, this fact can be directly proved by exploiting the pointwise picture of the transformation groupoid and the $I$-norm dominance of the reduced norm, which we already have in the twisted case of Inequality \ref{I-n0rm dominance}. We decide to keep the proof.
\begin{proof}
First note that for any $f \in C_c (\mathcal G, \omega),$ $$\|f\|_{\text{red}, \omega} \leq \|f\|_{I},$$ 
where $$\|f\|_I : = \max \biggr \{\sup\limits_{x \in X} \sum\limits_{t \in \Gamma} \left \lvert f (t, x) \right \rvert,\ \sup\limits_{x \in X} \sum\limits_{t \in \Gamma} \left \lvert f \left (t^{-1}, t \cdot x \right ) \right \rvert \biggr \}.$$ 
Thus, in order to show the property of rapid decay, we need to show that for all $f \in C_c (\mathcal G, \omega),$ $$\|f\|_{I} \leq C' \|f\|_{2, p, \ell},$$ for all $p > \frac {r} {2}$ and for some constant $C' > 0.$ Since $\mathcal G$ has the property of polynomial growth with growth exponent $r \geq 1,$ there exists $C > 0$ such that for all $n \geq 0,$ we have
$$\sup\limits_{x \in X} \left \lvert B_x (n) \right \rvert \leq C (1 + n)^{r}.$$ Now fix some $x_0 \in X$ and $p > \frac {r} {2}.$ Then we have 
\Bea \sum\limits_{t \in \Gamma} \left\lvert f (t, x_0) \right \rvert & = & \sum\limits_{t \in \Gamma} \left \lvert f (t, x_0) \right \rvert (1 + \ell (t, x_0))^{p} (1 + \ell (t, x_0))^{-p} \\ & \leq & \left (\sum\limits_{t \in \Gamma} \left \lvert f (t, x_0) \right \rvert^{2} (1 + \ell (t, x_0))^{2 p} \right )^{\frac {1} {2}} \left (\sum\limits_{t \in \Gamma} (1 + \ell (t, x_0))^{- 2 p} \right )^{\frac {1} {2}}
\Eea
Now
\Bea
\sum\limits_{t \in \Gamma} (1 + \ell (t, x_0))^{- 2 p} & = & 1 + \sum\limits_{t \in \Gamma \setminus \{e\}} (1 + \ell (t, x_0))^{- 2 p} \\ & = & 1 + \sum\limits_{n = 1}^{\infty} \sum\limits_{n - 1 < \ell (t, x_0) \leq n} (1 + \ell (t, x_0))^{- 2 p} \\ & \leq & 1 + \sum\limits_{n = 1}^{\infty} n^{-2 p} \left (\left \lvert B_{x_0} (n) \right \rvert - \left \lvert B_{x_0} (n - 1) \right \rvert \right ) 
\Eea
Now note that for any $k \in \mathbb N$ 
\Bea
\sum\limits_{n = 1}^{k} n^{-2 p} \left (\left \lvert B_{x_0} (n) \right \rvert - \left \lvert B_{x_0} (n - 1) \right \rvert \right ) & = & k^{-2 p} \left \lvert B_{x_0} (k) \right \rvert - 1 + \sum\limits_{n = 1}^{k - 1} \left (n^{-2 p} - (n + 1)^{-2 p} \right ) \left \lvert B_{x_0} (n) \right \rvert 
\Eea
Applying mean value theorem we can get hold of some $x_n \in (n, n + 1)$ for each $n \in \mathbb N$ such that $$n^{- 2 p} - (n + 1)^{- 2 p} = \frac {2 p} {x_n^{2 p + 1}} \leq \frac {2 p} {n^{2 p + 1}}.$$
Thus invoking polynomial growth we find that
\Bea
\sum\limits_{n = 1}^{k} n^{-2 p} \left (\left \lvert B_{x_0} (n) \right \rvert - \left \lvert B_{x_0} (n - 1) \right \rvert \right ) & \leq &  C k^{-2 p} (1 + k)^{r} - 1 + 2 p C \sum\limits_{n = 1}^{k - 1} \frac {(1 + n)^{r}} {n^{2 p + 1}} \\ & \leq & 2^{r} C k^{r - 2 p} - 1 + 2^{r + 1} p C \sum\limits_{n = 1}^{k - 1} \frac {1} {n^{2p - r + 1}}. 
\Eea
Since $p > \frac {r} {2},$ the infinite series $\sum\limits_{n = 1}^{\infty} \frac {1} {n^{2 p - r + 1}}$ converges and $k^{r - 2 p} \to 0$ as $k \to \infty.$ Now since the partial sums on the left are monotonically increasing, we can let $k \to \infty$ to obtain
\Bea
\sum\limits_{n = 1}^{\infty} n^{-2 p} \left (\left \lvert B_{x_0} (n) \right \rvert - \left \lvert B_{x_0} (n - 1) \right \rvert \right ) \leq 2^{r + 1} p C \sum\limits_{n = 1}^{\infty} \frac {1} {n^{2p - r + 1}} - 1 < \infty,
\Eea
for all $p > \frac {r} {2}.$ Therefore
\Bea
\sum\limits_{t \in \Gamma} (1 + \ell (t, x_0))^{- 2 p} & \leq & 2^{r + 1} p C \sum\limits_{n = 1}^{\infty} \frac {1} {n^{2p - r + 1}}.
\Eea
This shows that 
$$\sum\limits_{t \in \Gamma} \left \lvert f (t, x_0) \right \rvert \leq C' \left (\sum\limits_{t \in \Gamma} \left \lvert f (t, x_0) \right \rvert^{2} (1 + \ell (t, x_0))^{2 p} \right )^{\frac {1} {2}},$$
where $C' : = \sqrt {2^{r + 1} p C \sum\limits_{n = 1}^{\infty} \frac {1} {n^{2p - r + 1}}} < \infty.$
Since $x_0 \in X$ was arbitrary, it follows that
$$\sup\limits_{x \in X} \sum\limits_{t \in \Gamma} \left \lvert f (t, x) \right \rvert \leq C' \sup\limits_{x \in X} \left (\sum\limits_{t \in \Gamma} \left \lvert f (t, x) \right \rvert^{2} (1 + \ell (t, x))^{2 p} \right )^{\frac {1} {2}}.$$
By similar argument we can show that
$$\sup\limits_{x \in X} \sum\limits_{t \in \Gamma} \left \lvert f \left (t^{-1}, t \cdot x \right ) \right \rvert \leq C' \sup\limits_{x \in X} \left (\sum\limits_{t \in \Gamma} \left \lvert f \left (t^{-1}, t \cdot x \right ) \right \rvert^{2} (1 + \ell (t, x))^{2 p} \right )^{\frac {1} {2}}.$$
This shows that $\|f\|_{I} \leq C' \|f\|_{2, p, \ell},$ as required.

\end{proof}

\vspace{2mm}


\section{Compact quantum metric space structure on \texorpdfstring{$C^{\ast}$}{C*-}-algebra of twisted transformation groupoid} Let $\mathcal{G}=\Gamma\rtimes X$ be a transformation groupoid with a $2$-cocycle $\omega$. Let $\mathcal{G}$ admit a continuous proper length function $\ell$ such that $\mathcal{G}$ has $\omega$-twisted rapid decay property. Then we consider a class of CQMS structures on the twisted groupoid $C^{\ast}$-algebra $C^{\ast}_{r}(\mathcal{G},\omega)$ coming from stratified \clipnorms\ in the sense of \cite{austad2026quantum}*{Definition 3.7}. The class of CQMS structures is essentially borrowed from \cite{austad2026quantum}*{Theorem 3.14, Proposition 3.17}, where it is done for a general \`etale groupoid $C^{\ast}$-algebra. For a transformation groupoid $C^{\ast}$-algebra and its co-cycle twist the proof simplifies. First let us explain the class of stratified \clipnorms. To that end, We let $M_{\ell}$ denote the (possibly unbounded) operator on $L^2 (\mathcal G, \omega)$ of pointwise multiplication by $\ell.$ The following calculations have been done for the untwisted transformation groupoid $C^{\ast}$-algebras in \cite{joardar2025metrics}. The twisted case is an easy adaptation of the calculations done in \cite{joardar2025metrics}. We define the derivation by Dirac operator $D = M_{\ell}$ as follows: 
$$\Delta_{\omega} (f) = \left [D, \lambda^{\omega} (f) \right ] = \left [M_{\ell}, \lambda^{\omega} (f) \right ],$$ for all $f \in C_c (\mathcal G, \omega).$ For $k \in \mathbb N$ we define $\Delta^k (a)$ inductively by 
$$\Delta_{\omega}^k (f) = \left [D, \Delta_{\omega}^{k-1} (f) \right ]$$ 
for all $f \in C_c (\mathcal G, \omega).$ For any $f, \xi \in C_c (\mathcal G, \omega)$ and $(t,x) \in \mathcal G,$ adapting the same technique in \cite{joardar2025metrics}*{Lemma 2.3}, it follows that

\Bea
\Delta_{\omega} (f) (\xi) (t,x) & = & \sum\limits_{u \in \Gamma} \left (\ell (t, x) - \ell (u, x) \right ) f \left (t u^{-1}, u \cdot x \right ) \xi (u, x)\ \omega (t u^{-1}, u, x)
\Eea

\vspace{2mm}

A simple induction argument then shows the following: \begin{displaymath}\Delta_{\omega}^k (f) (\xi) (t,x) = \sum\limits_{u \in \Gamma} \left (\ell (t, x) - \ell (u, x) \right )^k f \left (t u^{-1}, u \cdot x \right ) \xi (u, x)\ \omega (t u^{-1}, u, x),\end{displaymath} for any $k \in \mathbb N.$

\vspace{2mm}



So for $f, \xi \in C_c (\mathcal G, \omega)$, noting that the cocycles take value in $\mathbb{T}$, adapting the argument in \cite{joardar2025metrics}*{Lemma 2.3}, one has the following estimate:
\begin{equation*}
 \left \|\Delta_{\omega}^k (f) (\xi) \right \|^2\leq \left \|M_{\ell}^k (f) \right \|_{I}^2 \|\xi\|^2.    
\end{equation*}

\vspace{5mm}

Therefore, $\left \|\Delta_{\omega}^k (f) \right \| \leq \left \|M_{\ell}^k (f) \right \|_{I},$ for any $f \in C_c (\mathcal G, \omega).$ This allows us to extend $\Delta^k (f)$ as a bounded operator on $L^2 (\mathcal G, \omega)$ as follows $:$ 
$$\Delta_{\omega}^k (f) (\xi) = \lim\limits_{n \to \infty} \Delta_{\omega}^k (f) (\xi_n)$$
where $\{\xi_n \}_{n \geq 1}$ is any sequence in $C_c (\mathcal G, \omega)$ converging to $\xi$ in the Hilbert module norm $\|\cdot\|.$ Next we will show that $\Delta_{\omega}^k (f)$ is adjointable for any $f \in C_c (\mathcal G, \omega).$

\begin{prop}\label{prop1}
 For any $f \in C_c (\mathcal G, \omega),$ $\Delta_{\omega}^k (f)$ is adjointable and 
 $$\Delta_{\omega}^k (f)^{\ast} = (-1)^k \Delta_{\omega}^k (f^{\ast_{\omega}}),$$ 
 where $f^{\ast_{\omega}} (t, x) = \overline {f \left (t^{-1}, t \cdot x \right ) \cdot \omega (t, t^{-1}, t \cdot x)},$ for all $(t, x) \in \mathcal G.$
\end{prop}

\begin{proof}
Since $C_c (\mathcal G, \omega)$ is dense in $L^2 (\mathcal G, \omega),$ in order to show the desired equality it is enough to show that for any $\xi, \eta \in C_c (\mathcal G, \omega),$ we have

$$\left \langle \left \langle \Delta_{\omega}^k (f) (\xi), \eta \right \rangle \right \rangle = (-1)^k \left \langle \left\langle \xi, \Delta_{\omega}^k (f^{\ast_{\omega}}) (\eta) \right \rangle \right \rangle.$$
We show the equality for $k = 1.$ For $f, \xi, \eta \in C_c (\mathcal G, \omega)$ and $x \in X$ we have

\Bea 
&& \left \langle \left \langle M_{\ell} (\lambda^{\omega} (f) (\xi)), \eta \right \rangle \right \rangle (x) \\ & = & \sum\limits_{t \in \Gamma} \ell (t,x) \overline {\lambda^{\omega} (f) (\xi) (t,x)}\ \eta (t, x) \\ & = & \sum\limits_{t \in \Gamma} \ell (t, x) \sum\limits_{u \in \Gamma} \overline {f \left (t u^{-1}, u \cdot x \right )}\ \overline {\xi (u, x)}\ \overline {\omega (t u^{-1}, u, x)}\ \eta (t,x) \\ & = & \sum\limits_{t \in \Gamma} \sum\limits_{u \in \Gamma} \ell (t, x) \overline {f \left (t u^{-1}, u \cdot x \right )}\ \overline {\xi (u, x)}\ \overline {\omega (t u^{-1}, u, x)}\ \eta (t,x) \\ & \stackrel{\text {Fubini}}{=} & \sum\limits_{u \in \Gamma} \sum\limits_{t \in \Gamma} \ell (t, x) \overline {f \left (t u^{-1}, u \cdot x \right )}\ \overline {\xi (u, x)}\ \overline {\omega (t u^{-1}, u, x)}\ \eta (t,x) \\ & = & \sum\limits_{u \in \Gamma} \overline {\xi (u, x)}\ \sum\limits_{t \in \Gamma} \ell (t, x) f^{\ast_{\omega}} \left (u t^{-1}, t \cdot x \right )\ \eta (t,x)\ \overline {\omega (t u^{-1}, u, x)}\ \omega (u t^{-1}, t u^{-1}, u \cdot x)  \\ & = & \sum\limits_{u \in \Gamma} \overline {\xi (u, x)}\ \sum\limits_{t \in \Gamma} f^{\ast_{\omega}} \left (u t^{-1}, t \cdot x \right )\ M_{\ell} (\eta) (t,x)\ \omega (u t^{-1}, t, x) \\ & = & \sum\limits_{u \in \Gamma} \overline {\xi (u, x)} \left (f^{\ast_{\omega}} \ast_{\omega} M_{\ell} (\eta) \right ) (u,x) \\ & = & \sum\limits_{u \in \Gamma} \overline {\xi (u, x)} \lambda^{\omega} (f^{\ast_{\omega}}) \left (M_{\ell} (\eta) \right ) (u, x) \\ & = & \left \langle \left \langle \xi, \lambda^{\omega} (f^{\ast_{\omega}}) \left (M_{\ell} (\eta) \right ) \right \rangle \right \rangle (x) \  
\Eea

and, 

\Bea
&& \left \langle \left \langle \lambda^{\omega} (f) \left (M_{\ell} (\xi) \right ), \eta \right \rangle \right \rangle (x) \\ & = & \sum\limits_{t \in \Gamma} \overline {\lambda^{\omega} (f) \left (M_{\ell} (\xi) \right ) (t, x)}\ \eta (t, x) \\ & = & \sum\limits_{t \in \Gamma} \overline {\left (f \ast_{\omega} M_{\ell} (\xi) \right ) (t, x)}\ \eta (t, x) \\ & = & \sum\limits_{t \in \Gamma} \sum\limits_{h \in \Gamma} \overline {f \left (t u^{-1}, u \cdot x \right )}\ \overline {M_{\ell} (\xi) (u, x)}\ \overline {\omega (t u^{-1}, u, x)}\ \eta (t, x) \\ & \stackrel{\text {Fubini}} {=} & \sum\limits_{h \in \Gamma} \sum\limits_{g \in \Gamma} \overline {f \left (t u^{-1}, u \cdot x \right )}\ \ell (u, x)\ \overline {\xi (u, x)}\ \overline {\omega (t u^{-1}, u, x)}\ \eta (t, x) \\ & = & \sum\limits_{u \in \Gamma} \ell (u, x) \overline {\xi (u, x)} \sum\limits_{t \in \Gamma} \overline {f \left (t u^{-1}, u \cdot x \right )}\ \eta (t, x)\ \overline {\omega (t u^{-1}, u, x)}\ \\ & = & \sum\limits_{u \in \Gamma} \ell (u, x) \overline {\xi (u, x)} \sum\limits_{t \in \Gamma} f^{\ast_{\omega}} \left (u t^{-1}, t \cdot x \right )\ \eta (t,x)\ \overline {\omega (t u^{-1}, u, x)}\ \omega (u t^{-1}, t u^{-1}, u \cdot x) \\ & = & \sum\limits_{u \in \Gamma} \ell (u, x) \overline {\xi (u, x)} \sum\limits_{t \in \Gamma} f^{\ast_{\omega}} \left (u t^{-1}, t \cdot x \right )\ \eta (t,x)\ \omega (u t^{-1}, t, x) \\ & = & \sum\limits_{u \in \Gamma} \ell (u,x) \overline {\xi (u, x)} \left (f^{\ast_{\omega}} \ast_{\omega} \eta \right ) (u, x)\ \\ & = & \sum\limits_{u \in \Gamma} \overline {\xi (u, x)} M_{\ell} \left (f^{\ast_{\omega}} \ast_{\omega} \eta \right ) (u, x) \\ & = & \left \langle \left \langle \xi, M_{\ell} \left (\lambda^{\omega} (f^{\ast_{\omega}}) (\eta) \right ) \right \rangle \right \rangle (x)
\Eea
Hence we have
\Bea 
\left \langle \left \langle \Delta_{\omega} (f) (\xi), \eta \right \rangle \right \rangle & = & \left \langle \left \langle M_{\ell} (\lambda^{\omega} (f) (\xi)), \eta \right \rangle \right \rangle - \left \langle \left \langle \lambda^{\omega} (f) (M_{\ell} (\xi)), \eta \right \rangle \right \rangle \\ & = & \left \langle \left \langle \xi, \lambda^{\omega} (f^{\ast_{\omega}}) \left (M_{\ell} (\eta) \right ) \right \rangle \right \rangle - \left \langle \left \langle \xi, M_{\ell} \left (\lambda^{\omega} (f^{\ast_{\omega}}) (\eta) \right ) \right \rangle \right \rangle \\ & = & - \left \langle \left \langle \xi, \Delta_{\omega} (f^{\ast_{\omega}}) (\eta) \right \rangle \right \rangle   
\Eea
This shows that $\Delta_{\omega} (f)^{\ast} = - \Delta_{\omega} (f^{\ast_{\omega}})$ for any $f \in C_c (\mathcal G, \omega),$ which proves the equality for the base case $k = 1.$ Then a routine induction argument establishes $\Delta_{\omega}^k (f)^{\ast} = (-1)^{k} \Delta_{\omega}^k (f^{\ast_{\omega}})$ for all $f \in C_c (\mathcal G, \omega),$ as required.
\end{proof}

\begin{defn}
Given a length function $\ell :\mathcal{G}\rightarrow[0,+\infty)$ on a transformation groupoid $\mathcal{G} = \Gamma \rtimes X,$ and for any $k \in \mathbb N,$ we define $L_{\ell}^{k}:C_{r}^{\ast}(\mathcal{G},\omega)\rightarrow[0,+\infty]$ by

$$L_{\ell}^k (f) : = \begin{cases} \left \|\Delta_{\omega}^k (f) \right \|_{\mathrm {adj}}, \quad f \in C_c (\mathcal G,\omega), \\ + \infty, \quad f \in C_r^{\ast} (\mathcal G,\omega) \setminus C_c (\mathcal G,\omega). \end{cases}$$
\end{defn}

\begin{cor}
For any $f \in C_c (\mathcal G, \omega)$ we have $$L_{\ell}^{k} (f) \geq \max \left \{\sup\limits_{x \in X} \left (\sum\limits_{t \in \Gamma} \ell (t, x)^{2 k} \left \lvert f (t, x) \right \rvert^{2} \right )^{\frac {1} {2}}, \sup\limits_{x \in X} \left (\sum\limits_{t \in \Gamma} \ell (t, x)^{2 k} \left \lvert f \left (t^{-1}, t \cdot x \right ) \right \rvert^{2} \right )^{\frac {1} {2}} \right \}.$$
\end{cor}

\begin{proof}
By invoking Proposition \ref{prop1}, we find that for all $f \in C_c (\mathcal G, \omega),$ $$L_{\ell}^{k} (f) = L_{\ell}^{k} \left (f^{\ast_{\omega}} \right ).$$ Since $\delta_e \in C_c (\mathcal G, \omega) \subseteq L^2 (\mathcal G, \omega)$ with module norm $\|\delta_e\| = 1,$ it follows that $$L_{\ell}^{k} (f) \geq \max \left \{\left \|\Delta_{\omega}^{k} (f) (\delta_e) \right \|, \left \|\Delta_{\omega}^{k} \left (f^{\ast_{\omega}} \right ) (\delta_e) \right \| \right \}.$$
Now it is easy to check that $$\left \|\Delta_{\omega}^{k} (f) (\delta_e) \right \| = \sup\limits_{x \in X} \left (\sum\limits_{t \in \Gamma} \ell (t, x)^{2 k} \left \lvert f (t, x) \right \rvert^{2} \right )^{\frac {1} {2}},$$ and, $$\left \|\Delta_{\omega}^{k} \left (f^{\ast_{\omega}} \right ) (\delta_e) \right \| = \sup\limits_{x \in X} \left (\sum\limits_{t \in \Gamma} \ell (t, x)^{2 k} \left \lvert f \left (t^{-1}, t \cdot x \right ) \right \rvert^{2} \right )^{\frac {1} {2}}.$$
Thus the desired result follows.

\end{proof}   
Now we are ready to discuss a class of CQMS structures on cocycle twists of transformation groupoid $C^{\ast}$-algebras. Let $\Gamma$ be a finitely generated discrete group acting on a compact metric space $(X, d).$ Consider the associated transformation groupoid $\mathcal G : = \Gamma \rtimes X$ endowed with a partition $\mathcal K : = \bigsqcup\limits_{t \in \Gamma} \{t\} \times X,$ where each $K_t : = \{t\} \times X \simeq X$ is equipped with a metric $d_t$ which is weakly equivalent to $d$ and $d \leq d_t$ for all $t \in \Gamma$ with $d_e = d.$ Following the notion of Austad (\cite{austad2026quantum}*{Definition 3.2}, we call it a metric stratification. Let $L_t$ be the \clipnorm\ on $C (X)$ induced by $d_t$ for all $t \in \Gamma.$ Given such a stratification $\mathcal K$ of $\Gamma,$ a length function $\ell : \mathcal G \rightarrow [0, \infty)$ and a $2$-cocycle $\omega : \Gamma \times X \times X \rightarrow \mathbb T,$ we define a countable collection of stratified Lipschitz seminorms $\left \{L^{\mathcal K, k}_{\ell} \right \}_{k = 1}^{\infty}$ on $C_r^{\ast} (\mathcal G, \omega)$ in the following way $:$
$$L^{\mathcal K, k}_{\ell} (f) = \begin{cases} \max \left \{L_{\ell}^{k} (f),\ \sup\limits_{t \in \Gamma} L_t (a_t) \right \}, \quad \mathrm {if}\ f = \sum\limits_{t \in \Gamma} \delta_t a_t \in C_c \left (\Gamma, \mathrm {Lip} (X) \right), \\ \infty, \quad \mathrm {otherwise}, \end{cases}$$
where $\mathrm {Lip} (X)$ is the set of all Lipschitz continuous functions on the compact metric space $(X, d).$

\begin{lem}
    $L^{\mathcal{K},k}_{\ell}$ is a sequence of Lipschitz seminorms on $C^{\ast}_{r}(\Gamma\rtimes X,\omega)$ for all $k.$
\end{lem}
\begin{proof}
     It is easy to show that $L^{\mathcal K, k}_{\ell}$ is a Lipschitz seminorm having one dimensional kernel generated by $\delta_{e} 1,$ where $e$ is the identity element in $\Gamma.$ To see the density of the domain of $L^{\mathcal K, k}_{\ell}$ in $C_r^{\ast} (\mathcal G, \omega),$ note that $d \leq d_t$ for all $t \in \Gamma$ and hence $L_t \leq L$ for all $t \in \Gamma.$ So the domain $\mathrm {Lip} (X)$ of $L$ is contained in the domain of $L_t$ for all $t \in \Gamma.$ Thus the domain of $L^{\mathcal K, k}_{\ell}$ is precisely $C_c (\Gamma, \mathrm {Lip} (X)).$ Since $\mathrm {Lip} (X)$ is dense in $C (X),$ it follows that the domain of $L^{\mathcal K, k}$ is dense in $C_r^{\ast} (\mathcal G, \omega).$
\end{proof}

\begin{rem}
If the underlying length function is clear from the context, we often suppress the index $\ell$ from the notation of the stratified Lipschitz seminorm and simply write it as $L^{\mathcal K, k}$ without worrying much about the length function.
\end{rem}

Next we show that in fact $L^{\mathcal{K},k}$'s are \clipnorm\  i.e. they metrize the weak$^{\ast}$-topolgy on the state space provided $C^{\ast}_{r}(\Gamma\rtimes X,\omega)$ satisfies RD with respect to $\ell.$ The proof of the following theorem is essentially an adaptation of the argument in \cite{austad2026quantum}*{Lemma 3.13, Proposition 3.17}. However, due to the fiberwise structure of the transformation groupoid, the analysis simplifies significantly, even though the metric stratification is more general in our setting. Hence we decide to include its proof in full generality in this paper for the sake of completeness and in order to keep the exposition self-contained. 
\begin{thm} \label {CQMS}
Let $\Gamma$ be a countable discrete group acting on a compact Hausdorff metric space $(X, d)$ and let $\mathcal G = \Gamma \rtimes X$ be the associated transformation groupoid endowed with a metric stratification and a $2$-cocycle $\omega : \Gamma \times \Gamma \times X \rightarrow \mathbb T.$ Let $\ell \colon \mathcal G \to [0,\infty)$ be a proper continuous length function such that $\mathcal G$ has the property of rapid decay with respect to $\ell$ with decay exponent $p > 0$ and constant $C > 0,$ then $\left (C_r^{\ast} (\mathcal G, \omega), L^{\mathcal K, k} \right )$ is a compact quantum metric space for all $k > p.$
\end{thm}
\begin{proof}
Let $\mu$ be a probability measure on $X$ and $E$ be the $C (X)$-valued conditional expectation on $C_r^{\ast} (\mathcal G, \omega),$ where $\mathcal G : = \Gamma \rtimes X.$ Consider the state $\sigma = \mu \circ E$ on $C_r^{\ast} (\mathcal G, \omega).$ In order to prove the result we will first show that for any finite subset $\Gamma_1 \subseteq \Gamma,$ the set $$E_{\mathcal K, k}^{\sigma} [\Gamma_1] : = \left \{f \in C_c (\mathcal G, \omega)\ :\ \sigma (f) = 0, L^{\mathcal K, k} (f) \leq 1\ \mathrm{and}\ \mathrm{supp} (f) \subseteq \bigsqcup\limits_{t \in \Gamma_1} K_t \right \}$$ is totally bounded with respect to the reduced norm, where $K_t = \{t\} \times X$ for all $t \in \Gamma.$ Fix some $\varepsilon > 0.$ We wish to show that $E_{\mathcal K, k}^{\sigma} [\Gamma_1]$ can be covered by finitely many $\varepsilon$-balls in the reduced norm. Define a metric $\tau_t$ on $X$ defined by $$\tau_t (x, y) = d_t ((t, x), (t, y)),$$ $x, y \in X,$ for all $t \in \Gamma.$ Then it is trivial to check that $\left (K_t, d_t \right ) \simeq \left (X, \tau_t \right ).$ Also since $d_t$ is weakly topologically equivalent to $d,$ we have $\left (X, \tau_t \right ) \simeq (X, d).$ This shows that $\left (K_t, d_t \right )$ is a compact metric space and hence totally bounded for all $t \in \Gamma.$ So for all $t \in \Gamma_1$ we have points $\left (t, x_{j}^{(t)} \right ) = \gamma_j^{(t)} \in K_t$ for $j = 1, \cdots, h_t$ such that $K_t \subseteq \bigcup\limits_{j = 1}^{h_t} B_{d_t} \left (\gamma_j^{(t)}, \frac {\varepsilon} {2 m} \right ),$ where $m = \left \lvert \Gamma_1 \right \rvert.$ Let $\left \{\rho_{j}^{(t)} \right \}_{1 \leq j \leq h_t}$ be the partition of unity subordinate to the cover $\left \{B_{d_t} \left (\gamma_{j}^{(t)}, \frac {\varepsilon} {2 m} \right ) \right \}_{1 \leq j \leq h_t},$ for all $t \in \Gamma_1.$ 

Now, for $f \in E_{\mathcal K, k}^{\sigma} [\Gamma_1],$ write $f = \sum\limits_{t \in \Gamma_1} \delta_t a_t$ and set $$\Phi_{\varepsilon} (f) : = \sum\limits_{t \in \Gamma_1} \delta_t \left (\sum\limits_{j = 1}^{h_t} a_t \left (x_j^{(t)} \right ) \rho_j^{(t)} (t, \cdot) \right ).$$
Then we have 
\Bea
\left \|f - \Phi_{\varepsilon} (f) \right \|_{\mathrm {red}} & = & \left \|\sum\limits_{t \in \Gamma_1} \delta_t a_t - \sum\limits_{t \in \Gamma_1} \delta_t \left (\sum\limits_{j = 1}^{h_t} a_t \left (x_j^{(t)} \right ) \rho_{j}^{(t)} (t, \cdot) \right ) \right \|_{\mathrm {red}} \\ & \leq & \sum\limits_{t \in \Gamma_1} \left \|a_t - \sum\limits_{j = 1}^{h_t} a_t \left (x_j^{(t)} \right ) \rho_{j}^{(t)} (t, \cdot) \right \|_{\infty} \\ & = & \sum\limits_{t \in \Gamma_1} \sup\limits_{x \in X} \left \lvert a_t (x) - \sum\limits_{j = 1}^{h_t} a_t \left (x_j^{(t)} \right ) \rho_{j}^{(t)} (t, x) \right \rvert \\ & \leq & \sum\limits_{t \in \Gamma_1} \sup\limits_{x \in X} \sum\limits_{j = 1}^{h_t} \left \lvert a_t (x) - a_t \left (x_j^{(t)} \right ) \right \rvert \rho_{j}^{(t)} (t, x) \\ & \leq & \sum\limits_{t \in \Gamma_1} \sup\limits_{x \in X} \sum\limits_{j = 1}^{h_t} L_t \left (a_t \right ) d_t \left ((t, x), (t, x_j^{(t)}) \right ) \rho_j^{(t)} (t, x) \\ & \leq & m \cdot \frac {\varepsilon} {2 m} \cdot L^{\mathcal K, k} (f) \\ & \leq & \frac {\varepsilon} {2} \\ & < & \varepsilon. 
\Eea
Thus we can arrange that $E_{\mathcal K, k}^{\sigma} [\Gamma_1]$ is less than $\varepsilon$ away in the reduced norm from its image under $\Phi_{\varepsilon}$ in the finite dimensional subspace of $C_r^{\ast} (\mathcal G, \omega)$ spanned by $\left ( \left (\rho_j^{(t)} \right )_{j = 1}^{h_t} \right )_{t \in \Gamma_1}.$ Thus in order to show that $E_{\mathcal K, k}^{\sigma} [\Gamma_1]$ is totally bounded in the reduced norm it is enough to show that the image of $\Phi_{\varepsilon}$ is bounded in the reduced norm.

\vspace{2mm}

For this purpose, let $f \in E_{\mathcal K, k}^{\sigma} [\Gamma_1],$ and $$\Phi_{\varepsilon} (f) = \sum\limits_{t \in \Gamma_1} \delta_t \left (\sum\limits_{j = 1}^{h_t} a_t \left (x_j^{(t)} \right ) \rho_{j}^{(t)} (t, \cdot ) \right ).$$ We need to show that there is a uniform bound of the reduced norm of those elements. First, let us obtain an upper bound.

$$\left \|\Phi_{\varepsilon} (f) \right \|_{\mathrm {red}} \leq \sum\limits_{t \in \Gamma_1} \sum\limits_{j = 1}^{h_t} \left \lvert a_t \left (x_j^{(t)} \right ) \right \rvert \left \|\rho_j^{(t)} (t, \cdot) \right \|_{\infty} \leq m \left (\max\limits_{t \in \Gamma_1} h_t \right ) \left (\max\limits_{t, j} a_t \left (x_{j}^{(t)} \right ) \right ).$$
We first consider $t \neq e.$ Since $\ell$ is continuous, it follows that $$q : = \inf\limits_{\substack {t \in \Gamma_1 \setminus \{e\} \\ x \in X}} \ell (t, x) > 0.$$ Let $A_k = \frac {1} {q^k} > 0.$ Then for all $t \in \Gamma_1$ we have
\Bea
\left \lvert a_t \left (x_j^{(t)} \right ) \right \rvert &  \leq & \left \|a_t \right \|_{\infty} \\ & \leq & A_k\ \sup_{x \in X} \left \|\delta_t a (\cdot, x) \cdot \ell^{k} (\cdot, x) \right \|_{\ell^{2} (\Gamma)} \\ & \leq & A_k\ \sup_{x \in X} \left \|f (\cdot, x) \cdot \ell^{k} (\cdot, x) \right \|_{\ell^{2} (\Gamma)} \\ & \leq & A_k\ L_{\ell}^{k} (f) \\ & \leq & A_k\ L^{\mathcal K, k} (f) \\ & \leq & A_k.
\Eea
This covers the case $t \neq e.$ For $t = e,$ note that $\sigma (f) = \int_{X} a_e\ d \mu.$ Letting $\mu$ to be Dirac mass at some $x_e \in X,$ it follows that $a \left (x_e \right ) = 0.$ Let $\left \|a_e  \right \|_{\infty} = \left \lvert a_e \left (x_1 \right ) \right \rvert$ for some $x_1 \in X.$ Thus for all $1 \leq j \leq h_0$ we have \Bea \left \lvert a_e \left (x_j^{(e)} \right ) \right \rvert & \leq & \left \|a_e \right \|_{\infty} \\ & = & \left \lvert a_e \left (x_1 \right ) \right \rvert \\ & = & \left \lvert a_e \left (x_1 \right ) - a_e \left (x_e \right ) \right \rvert \\ & \leq & d \left (x_e, x_1 \right ) L \left (a_e \right ) \\ & \leq & \mathrm {diam} (X, d)\ L^{\mathcal K, k} (f) \\ & \leq & \mathrm {diam} (X, d).  \Eea
This covers the case $t = e.$ So the image of $\Phi_{\varepsilon}$ is bounded, as required.

\vspace{2mm}

Now in order to prove the result it is enough to show that the set
$$E_{\mathcal K, k}^{\sigma} : = \left \{f \in C_c (\mathcal G, \omega)\ :\ L^{\mathcal K, k} (f) \leq 1\ \mathrm {and}\ \sigma (f) = 0 \right \}$$ is totally bounded in the reduced norm. Let us get hold of $n_0 \in \mathbb N$ such that $$2^{p} n_0^{p - k} < \frac {\varepsilon} {2 C}.$$ By compactness of $B_{\ell} \left (n_0 \right )$ (compactness follows from properness and continuity of $\ell$) there exists a finite subset $\Gamma' \subseteq \Gamma$ with $\left \lvert \Gamma' \right \rvert = m'$ (say) such that $$B_{\ell} \left (n_0 \right ) \subseteq \bigsqcup\limits_{t \in \Gamma'} K_t.$$ Let $h \in E_{\mathcal K, k}^{\sigma}$ with $h = \sum\limits_{t \in \Gamma} \delta_t b_t.$ Then $$\left \|\sum\limits_{t \in \Gamma \setminus \Gamma'} \delta_t b_t \right \|_{\mathrm {red}} < \frac {\varepsilon} {2}.$$ Consider the set $$E_{\mathcal K, k}^{\sigma} [\Gamma'] : = \left \{f \in C_c (\mathcal G, \omega)\ :\ L^{\mathcal K, k} (f) \leq 1, \sigma (f) = 0\ \mathrm{and}\ \mathrm{supp} (f) \subseteq \bigsqcup\limits_{t \in \Gamma'} K_t \right \}.$$ Then by first part of the proof, it follows that $E_{\mathcal K, k}^{\sigma} [\Gamma']$ is totally bounded with respect to the reduced norm. So every element of $E_{\mathcal K, k}^{\sigma} [\Gamma']$ can be covered by finitely many balls of radius $\frac {\varepsilon} {2}.$ Thus in order to demonstrate that every element of $E_{\mathcal K, k}^{\sigma}$ can be covered by finitely many balls of radius $\varepsilon,$ it is enough to show that for any $f = \sum\limits_{t \in \Gamma} \delta_t a_t \in C_c (\mathcal G, \omega)$ with $L_{\ell}^{k} (f) \leq 1,$ we have $$L_{\ell}^{k} \left (\sum\limits_{t \in \Gamma'} \delta_t a_t \right ) \leq C',$$ for some constant $C' > 0.$
First note that $$L_{\ell}^{k} (g) \leq \left \|M_{\ell}^{k} (g) \right \|_{I},$$ for all $g \in C_c (\mathcal G, \omega).$ So 
\Bea L_{\ell}^{k} \left (\sum\limits_{t \in \Gamma'} \delta_t a_t \right ) & \leq & \max \left \{\sup\limits_{x \in X} \sum\limits_{t \in \Gamma'} \ell (t, x)^{k} \left \lvert a_t (x) \right \rvert, \sup\limits_{x \in X} \sum\limits_{t \in \Gamma'} \ell (t, x)^{k} \left \lvert a_{t^{-1}} (t \cdot x) \right \rvert \right \} \\ & \leq & \sqrt {m'} \max \left \{\sup\limits_{x \in X} \left (\sum\limits_{t \in \Gamma'} \ell (t, x)^{2 k} \left \lvert a_t (x) \right \rvert^{2} \right )^{\frac {1} {2}}, \sup\limits_{x \in X} \left (\sum\limits_{t \in \Gamma'} \ell (t, x)^{2 k} \left \lvert a_{t^{-1}} (t \cdot x) \right \rvert^{2} \right )^{\frac {1} {2}}  \right \} \\ & \leq & \sqrt {m'} \max \left \{\sup\limits_{x \in X} \left (\sum\limits_{t \in \Gamma} \ell (t, x)^{2 k} \left \lvert a_t (x) \right \rvert^{2} \right )^{\frac {1} {2}}, \sup\limits_{x \in X} \left (\sum\limits_{t \in \Gamma} \ell (t, x)^{2 k} \left \lvert a_{t^{-1}} (t \cdot x) \right \rvert^{2} \right )^{\frac {1} {2}}  \right \} \\ & \leq & \sqrt {m'}\ L_{\ell}^{k} (f) \\ & \leq & \sqrt {m'}.
\Eea
This completes the proof.
\end{proof}

Henceforth, we shall call $L^{\mathcal K, k}$ a stratified \clipnorm.
\begin{rem}\label{noleibnitz}
    The sequence of the \clipnorm\ discussed above does not satisfy the Leibnitz property. Leibnitz property is important from the perspective of quantum Gromov-hausdorff propinquity and product entropy. But lack of it does not play any significant role in studying metric dimension.
\end{rem}
\subsection{Bounds of metric dimension for polynomial growth} In this subsection, we obtain bounds of the metric dimension of $2$-cocycle twisted transformation groupoid $C^{\ast}$-algebras admitting a length function of polynomial growth. As usual, let $\Gamma$ be a finitely generated discrete group acting on a compact metric space $(X,d)$. Recall that if a transformation groupoid $\mathcal{G}=\Gamma\rtimes X$ admitting a $2$-cocycle $\omega$ has polynomial growth property with growth exponent $r\geq 1$ for some continuous proper length function $\ell$, then it has the twisted RD-property with exponent $p>\frac{r}{2}$ (Lemma \ref{polytorapid}). Then for any choice of metric stratification $\mathcal{K}$, the pair $(C^{\ast}_{r}(\mathcal{G},\omega),L^{\mathcal{K},k})$ is a CQMS for any $k>\frac{r}{2}$ by Theorem \ref{CQMS}. We shall consider two types of stratifications. We call them static and dynamic stratifications. 
 \subsubsection{Static stratification} \label{static_stratification}  A static metric stratification corresponds to the uniform choice $d$ on each fiber $\{t\}\times X$. We denote the static metric stratification by $\mathcal{K}_{S}$ and the corresponding stratified \clipnorm\ by $L^{\mathcal{K}_{S},k}$ for any $k$. 
\begin{thm} \label{general cocycle bound}
Let $\mathcal{G}$ be the transformation groupoid associated with an action of $\Gamma$ on a compact metric space $(X, d)$ equipped with a normalized $2$-cocycle $\omega : \Gamma \times \Gamma \times X \rightarrow \mathbb T.$ Let $\ell$ be a continuous proper length function on $\mathcal G$ satisfying the condition \eqref{Lip-length}, under which $\mathcal G$ possesses the property of polynomial growth of degree $r \geq 1.$ Let $C_r^{\ast}(\mathcal {G}, \omega)$ be equipped with the static stratified \clipnorm\ $L^{\mathcal{K}_{S}, k}$. Then for all $k > \frac{r}{2}$, we have
\begin{displaymath}\mathrm{Mdim}_{L^{\mathcal{K}_{S},k}} (C_r^{\ast}(\mathcal{G}, \omega)) \leq \frac {2 r} {2 k - r} + M_X \cdot \left (\frac {2 k + r} {2 k - r} \right ), \end{displaymath}
where $M_X = \dim_B (X, d),$ the Kolmogorov dimension of $(X,d)$.
\end{thm}

\begin{proof}
Let $f = \sum\limits_{t \in \Gamma} \delta_t a_t \in C_c(\Gamma, \mathrm {Lip} (X))$ such that $L^{\mathcal{K}_S, k}(f) \leq 1.$ and $\lvert\lvert f\rvert\rvert_{\omega, \mathrm {red}}\leq 1$. We have exactly three simultaneous bounds $:$
\begin{enumerate}
\item $\|f\|_{\omega, \mathrm {red}} \leq 1.$ By inequality \ref{I-n0rm dominance}, this condition ensures that $$\|a_t\|_{\infty} \leq \|f\|_{\infty} \leq \|f\|_{\omega, \mathrm {red}} \leq 1,$$ for all $t \in \Gamma.$
\item $L_{\ell}^k\left (\sum\limits_{t \in \Gamma} \delta_t a_t \right ) \leq 1.$
 \item $L_X(a_t) \leq 1$ for all $t \in \Gamma,$ meaning every coefficient function uniformly belongs to $\mathcal{L}_1^X$, the standard unit Lipschitz ball in $C(X).$
\end{enumerate}
Fix an arbitrary $\delta > 0.$ Since $\mathcal G$ satisfies the property of rapid decay with decay exponent $p > \frac {r} {2},$ fixing some $p \in \left (\frac {r} {2}, k \right ),$ we define $n_0$ to be the \textit{smallest positive integer} such that
\begin{equation} \label{inequality} 2^p C n_0^{p-k} < \frac{\delta}{2}, \end{equation}
where $C > 0$ is the constant arising from the property of rapid decay. Let $m_0$ be the smallest positive integer for which the following implication holds $:$
$$\ell_{\Gamma} (t) > m_0 \implies \ell (t, x) > n_0,$$ for all $x \in X.$ Then, letting $f_{> m_0} : = \sum\limits_{\ell_{\Gamma} (t) > m_0} \delta_t a_t,$ we have
$$\|f_{> m_0}\|_{\omega,\mathrm{red}} = \left\| \sum_{\ell_{\Gamma} (t) > m_0} \delta_t a_t \right\|_{\omega, \mathrm{red}} \le 2^p C n_0^{p-k} < \frac{\delta}{2}.$$
Consequently, any $\frac{\delta}{2}$-approximation of the truncated element $f_{\leq m_0} : = \sum\limits_{\ell_{\Gamma} (t) \leq m_0} \delta_t a_t$ yields a valid $\delta$-approximation of $f.$ Since $n_0$ is defined as the smallest integer satisfying \eqref{inequality}, we have
$$\left( \frac{C \cdot 2^{p+1}}{\delta} \right)^{\frac{1}{k-p}} < n_0 \le \left( \frac{C \cdot 2^{p+1}}{\delta} \right)^{\frac{1}{k-p}} + 1.$$
Therefore, by the virtue of Lemma \ref{length-Lip} it follows that, for sufficiently small $\delta > 0$ 
\begin{equation} \label{bound for m_0} m_0 \leq \left( \frac{C \cdot 2^{p+1}}{\delta} \right)^{\frac{1}{k-p}} + 2 + CD < 2 \left( \frac{C \cdot 2^{p+1}}{\delta} \right)^{\frac{1}{k-p}} = 2 A \delta^{- \frac {1} {k - p}}, \end{equation}
where $A : = \left (C \cdot 2^{p + 1} \right )^{\frac {1} {k - p}}.$

Since $\Gamma$ has polynomial growth of degree $r \geq 1$ with respect to $\ell_{\Gamma},$ the cardinality $V \left (m_0 \right )$ of the length ball $B_{\ell_{\Gamma}} (m_0)$ is bounded by $K m_0^{r}$ for some constant $K > 0$. We define $\delta'$ as follows $:$
$$\delta' := \frac{\delta}{2 K m_0^{r}}.$$
Let $\mathcal L_1^{X} : = \left \{a \in C(X)\ :\ L(a) \leq 1, \|a\|_{\infty} \leq 1 \right \}.$
By definition, for this specific $\delta' > 0$, there exists finite-dimensional $Z \subseteq C(X)$ which $\delta'$-approximate $\mathcal{L}_1^X$ and
$$\dim(Z) = D(\mathcal{L}_1^X, \delta').$$
Furthermore, $a_t \in \mathcal{L}_1^X$. We can therefore choose elements $z_t \in Z$ satisfying
$$\|a_t - z_t\|_{\infty} < \delta'.$$
Consider the finite-dimensional subspace $W = \mathrm{span}\{ y \cdot h \mid y \in B_{\ell_{\Gamma}} \left (m_0 \right ), h \in Z \} \subseteq C_r^*(\mathcal{G}, \omega).$ Then
$$\dim(W) \leq V \left (m_0 \right ) \cdot \dim(Z) \leq K m_0^{r} \cdot D(\mathcal{L}_1^X, \delta').$$ Then we have
\Bea
    \left\|f_{\le m_0} - \sum\limits_{\ell_{\Gamma} (t) \leq m_0} \delta_t z_t \right\|_{\omega,\mathrm{red}} & \leq & \sum\limits_{\ell_{\Gamma} (t) \le m_0} \left \|a_t - z_t \right \|_{\infty} \\
    & \leq & V(m_0) \delta' \\ & \leq & K m_0^{r} \delta' \\ & = & \frac {\delta} {2}.
\Eea    
Combining this with the tail bound, the subspace $W$ perfectly $\delta$-approximates the unit Lip-ball $\mathcal L_1^{\mathcal{K_S}, k}.$ Thus, $D(\mathcal {L}_1^{\mathcal{K_S}, k}, \delta) \le \dim(W)$. Taking the logarithm and dividing both sides by $\log \delta^{-1},$ we have
\bea \label{eqn:dim-bound}
\frac{\log D(\mathcal{L}_1^{\mathcal{K_S}, k}, \delta)}{\log \delta^{-1}} & \leq & \frac{\log K m_0^{r}}{\log \delta^{-1}} + \frac{\log D(\mathcal{L}_1^X, \delta')}{\log \delta^{-1}} \nonumber \\
& = & \frac{\log K} {\log \delta^{-1}} + r\ \frac{\log m_0}{\log \delta^{-1}} + \frac{\log D(\mathcal{L}_1^X, \delta')}{\log \delta'^{-1}} \cdot \frac{\log \delta'^{-1}}{\log \delta^{-1}}
\eea
Let us now evaluate the limit supremum as $\delta \to 0^+.$ Taking logarithms in both sides of \eqref{bound for m_0} we have
$$\log m_0 \leq \log(2A) + \frac{1}{k-p} \log \delta^{-1}.$$
Substituting this into the expansion of $\log \delta'^{-1} = \log(2 K) + \log \delta^{-1} + r \log m_0$ yields
\Bea
\log \delta'^{-1} & \leq & \log(2 K) + \log \delta^{-1} + r \left[ \log(2A) + \frac{1}{k-p} \log \delta^{-1} \right] \\
& = & B + \left( 1 + \frac{r} {k - p} \right) \log \delta^{-1},
\Eea
where $B : = \log(2 K) + r \log(2A)$ is a constant independent of $\delta,$ since so are $K$ and $A.$ Dividing both sides by $\log \delta^{-1}$ and taking the limit supremum as $\delta \to 0^+$ we thus have
$$\limsup\limits_{\delta \to 0^{+}} \frac {\log m_0} {\log \delta^{-1}} \leq \frac {1} {k - p},$$ and,
$$\limsup_{\delta \to 0^+} \frac{\log \delta'^{-1}}{\log \delta^{-1}} \le 1 + \frac{r} {k - p}.$$
Finally, as $\delta \to 0^+$, we necessarily have $\delta' \to 0^+.$
Substituting these evaluated limits in \eqref{eqn:dim-bound}, we can conclude that
$$\mathrm{Mdim}_{L^{\mathcal{K_S}, k}}(C_r^*(\mathcal{G}, \omega)) \leq \frac {r} {k - p} + M_X \cdot \left( 1 + \frac{r} {k - p} \right).$$
Since this holds for any $p > \frac {r} {2},$ letting $p \to \frac {r} {2}^{+},$ the desired result follows.
\end{proof}

\begin{thm}\label{betterbound}
Let $\mathcal{G}$ be the transformation groupoid associated with an action of $\Gamma$ on a compact metric space $(X, d)$ equipped with a normalized $2$-cocycle $\sigma : \Gamma \times \Gamma \rightarrow \mathbb T.$ Let $\ell$ be a continuous proper length function on $\mathcal G$ satisfying the condition \eqref{Lip-length}, under which $\mathcal G$ possesses the property of polynomial growth of degree $r \geq 1.$ Let $C_r^{\ast}(\mathcal {G}, \sigma)$ be equipped with the static stratified \clipnorm\ $L^{\mathcal{K}_{S}, k}$. Then for all $k > \frac{r}{2}$, we have
\begin{displaymath}\mathrm{Mdim}_{L^{\mathcal{K}_{S},k}} (C_r^{\ast}(\mathcal{G}, \sigma)) \leq  \frac {4 k + r} {2 k - r} \left( M_{\Gamma} + M_X \right), \end{displaymath}
where $M_{\Gamma}  = \mathrm {Mdim}_{L_{\ell_{\Gamma}}^{k}} \left (C_r^{\ast} (\Gamma, \sigma) \right )$ and $M_X = \dim_B (X, d).$ 
\end{thm}

\begin{proof}
Let $f = \sum\limits_{t \in \Gamma} \delta_t f_t \in C_c(\Gamma, \mathrm {Lip} (X))$ such that $L^{\mathcal{K}_S, k}(f) \leq 1.$ Fix some $p \in \left (\frac {r} {2}, k \right )$ and a $\delta > 0$ sufficiently small. Then adapting the same technique as in the proof of Theorem \ref{general cocycle bound}, we can get hold of some $m_0 \in \mathbb N$ such that
$$\left \|f_{> m_0} \right \|_{\sigma, \mathrm {red}} = \left \|\sum\limits_{\ell_{\Gamma} (t) > m_0} \delta_t a_t \right \|_{\sigma, \mathrm {red}} < \frac {\delta} {2},$$ with
\begin{equation} \label{upper bound for m_0} m_0 < 2 A \delta^{- \frac {1} {k - p}}, \end{equation}
where $A : = \left (C \cdot 2^{p + 1} \right )^{\frac {1} {k - p}}.$ Consequently, any $\frac{\delta}{2}$-approximation of the truncated element $f_{\leq m_0} : = \sum\limits_{\ell_{\Gamma} (t) \leq m_0} \delta_t a_t$ yields a valid $\delta$-approximation of $f.$

\vspace{0.5mm}

Since $\Gamma$ has polynomial growth of degree $r \geq 1$ with respect to $\ell_{\Gamma},$ the cardinality $V \left (m_0 \right )$ of the length ball $B_{\ell_{\Gamma}} (m_0)$ is bounded by $K m_0^{r}$ for some constant $K > 0.$ We assume without loss of generality that $\delta < 1$ and $K > 1$ to ensure that $\frac{\delta}{6 K m_0^{k + r}} < 1$. We define $\delta'$ as follows $:$
$$\delta' := \frac{\delta}{6 K m_0^{k + r}}.$$
Let $$\mathcal L_1^{\Gamma} : = \left \{f \in C_c (\Gamma, \sigma)\ :\ L_{\ell_{\Gamma}}^{k} (f) \leq 1, \|f\|_{\mathrm {red}} \leq 1 \right \},$$
and, $$\mathcal L_1^{X} : = \left \{a \in C(X)\ :\ L(a) \leq 1, \|a\|_{\infty} \leq 1 \right \}.$$
By definition, for this specific $\delta' > 0$, there exist finite-dimensional subspaces $Y \subseteq C_r^*(\Gamma, \sigma)$ and $Z \subseteq C(X)$ that $\delta'$-approximate $\mathcal{L}_1^{\Gamma}$ and $\mathcal{L}_1^X$ respectively, satisfying
$$\dim(Y) = D(\mathcal{L}_1^{\Gamma}, \delta') \quad \text{and} \quad \dim(Z) = D(\mathcal{L}_1^X, \delta').$$
For each $t \in \Gamma$ with $\ell_{\Gamma} (t) \le m_0$, we have $L_{\ell_{\Gamma}}^k(\delta_t) = (\ell_{\Gamma} (t))^k \le m_0^k$, which implies $\frac{\delta_t}{m_0^k} \in \mathcal{L}_1^{\Gamma}.$ Furthermore, $a_t \in \mathcal{L}_1^X$. We can therefore choose elements $y_t \in Y$ and $z_t \in Z$ satisfying
$$\|\delta_t - m_0^k y_t\|_{\sigma, \mathrm{red}} < m_0^k \delta' \quad \text{and} \quad \|a_t - z_t\|_{\infty} < \delta'.$$
Consider the finite-dimensional subspace $W = \mathrm{span}\{ y \cdot h \mid y \in Y, h \in Z \} \subseteq C_r^*(\mathcal{G}, \sigma)$. Then
$$\dim(W) \le \dim(Y) \cdot \dim(Z) = D(\mathcal{L}_1^{\Gamma}, \delta') \cdot D(\mathcal{L}_1^X, \delta').$$
Approximating the $t$-th term $\delta_t a_t$ using $(m_0^k y_t) z_t \in W,$ the triangle inequality yields
\Bea
    \| \delta_t a_t - (m_0^k y_t) z_t \|_{\sigma, \mathrm{red}} & \leq & \| \delta_t a_t - \delta_t z_t \|_{\sigma, \mathrm{red}} + \|\delta_t z_t - (m_0^k y_t) z_t \|_{\sigma, \mathrm{red}} \\
    & = & \| \delta_t (a_t - z_t) \|_{\mathrm{red}} + \| (\delta_t - m_0^k y_t) z_t \|_{\mathrm{red}}.
\Eea
Since $\delta_t$ is unitary, the first term is bounded by $\|a_t - z_t\|_{\infty} < \delta'$. For the second term, since $\|a_t\|_{\infty} \le 1$, the triangle inequality ensures $\|z_t\|_{\infty} \le \|a_t\|_{\infty} + \|a_t - z_t\|_{\infty} < 1 + \delta'$. Substituting these bounds gives
$$\| (\delta_t - m_0^k y_t) z_t \|_{\sigma, \mathrm{red}} \le \| \delta_t - m_0^k y_t \|_{\sigma, \mathrm{red}} \| z_t \|_{\infty} < (m_0^k \delta')(1 + \delta').$$
Summing over all $t \in \Gamma$ with $\ell_{\Gamma} (t) \leq n_0$ we have
\Bea
    \left\| f_{\le m_0} - \sum_{\ell_{\Gamma} (t) \leq m_0} m_0^k y_t z_t \right\|_{\sigma, \mathrm{red}} & \leq & \sum_{\ell_{\Gamma} (t) \le m_0} \Big[ \delta' + m_0^k \delta'(1 + \delta') \Big] \\
    & < & V(m_0) \Big[ \delta' + m_0^k \delta'(1 + \delta') \Big].
\Eea
Since $\delta' < 1$, the bracketed term is strictly bounded by $3 m_0^k \delta'.$ Thus 
$$V(m_0) \left( 3 m_0^k \delta' \right) \le K m_0^{r} \left( 3 m_0^k \frac{\delta}{6 K m_0^{k + r}} \right) = \frac{\delta}{2}.$$
Combined with the tail bound, the subspace $W$ perfectly $\delta$-approximates the unit Lip-ball $\mathcal L_1^{\mathcal{K}_S, k}$. Thus, $D(\mathcal {L}_1^{\mathcal{K}_S, k}, \delta) \le \dim(W)$. Taking the logarithm and dividing both sides by $\log \delta^{-1},$ we have
\bea \label{eq:dim-bound}
\frac{\log D(\mathcal{L}_1^{\mathcal{K}_S, k}, \delta)}{\log \delta^{-1}} & \leq & \frac{\log D(\mathcal{L}_1^{\Gamma}, \delta')}{\log \delta^{-1}} + \frac{\log D(\mathcal{L}_1^X, \delta')}{\log \delta^{-1}} \nonumber \\
& = & \left[ \frac{\log D(\mathcal{L}_1^{\Gamma}, \delta')}{\log \delta'^{-1}} + \frac{\log D(\mathcal{L}_1^X, \delta')}{\log \delta'^{-1}} \right] \frac{\log \delta'^{-1}}{\log \delta^{-1}}.
\eea
Let us now evaluate the limit supremum as $\delta \to 0^+$. Taking logarithms in both sides of \eqref{upper bound for m_0}, we have
$$\log m_0 \leq \log(2A) + \frac{1}{k-p} \log \delta^{-1}.$$
Substituting this into the expansion of $\log \delta'^{-1} = \log(6K) + \log \delta^{-1} + (k + r) \log m_0$ yields
\Bea
\log \delta'^{-1} & \leq & \log(6 K) + \log \delta^{-1} + (k + r) \left[ \log(2A) + \frac{1}{k-p} \log \delta^{-1} \right] \\
& = & B + \left( 1 + \frac{k + r}{k - p} \right) \log \delta^{-1},
\Eea
where $B : = \log(6K) + (k + r)\log(2A)$ is a constant independent of $\delta$. Dividing by $\log \delta^{-1}$ and taking the limit supremum as $\delta \to 0^+,$ we have
$$\limsup_{\delta \to 0^+} \frac{\log \delta'^{-1}}{\log \delta^{-1}} \le 1 + \frac{k + r}{k - p}.$$
Finally, as $\delta \to 0^+$, we necessarily have $\delta' \to 0^+,$ Then each summand in \eqref{eq:dim-bound} evaluates to
$$\limsup\limits_{\delta' \to 0^+} \frac{\log D(\mathcal{L}_1^{\Gamma}, \delta')}{\log \delta'^{-1}} = M_{\Gamma},$$
and,
$$\limsup\limits_{\delta' \to 0^+} \frac{\log D(\mathcal{L}_1^X, \delta')}{\log \delta'^{-1}} = M_X.$$
Substituting these evaluated limits in \eqref{eq:dim-bound}, we can conclude that
$$\mathrm{Mdim}_{L^{\mathcal{K}_S, k}}(C_r^*(\mathcal{G}, \sigma)) \le \left( M_{\Gamma} + M_X \right) \left( 1 + \frac{k + r}{k - p} \right ).$$
Since this holds for any $p > \frac {r} {2},$ letting $p \to \frac {r} {2}^{+},$ the desired result follows.
\end{proof}
Now we shall give a lower bound of the metric dimension for the twisted transformation groupoid $C^{\ast}$-algebra. To establish lower bound we are going to use the following theorem due to Voiculescu (\cite{Voiculescu}*{Lemma 7.8}).
\begin{thm}(Voiculescu's Lemma)
\label{Voiculescu} Let $X$ be an orthonormal set in a Hilbert space $H$. Then for any $\delta>0$, $D(X,\delta)\geq(1-\delta^{2})\ {\mathrm {Card}}(X).$
\end{thm}

\begin{thm}  \label{Mdim lower bound}
Let $\mathcal{G}$ be the transformation groupoid associated with an action of $\Gamma$ on a compact metric space $(X, d)$ equipped with a normalized $2$-cocycle $\omega : \Gamma \times \Gamma \times X \rightarrow \mathbb T.$ Let $\ell$ be a continuous proper length function on $\mathcal G$ satisfying the condition \eqref{Lip-length}, under which $\mathcal G$ possesses the property of rapid decay with decay exponent $p > 0.$ Let $C_r^{\ast}(\mathcal {G}, \omega)$ be equipped with the static stratified \clipnorm\ $L^{\mathcal{K}_{S}, k}.$ Then for all $k > p,$ we have
$$\mathrm {Mdim}_{L^{\mathcal K_S, k}} \left (C_r^{\ast} (\mathcal G, \omega) \right) \geq \dim_B (X) + \frac{1}{k}.$$
\end{thm}

\begin{proof}
Fix a $\delta > 0$. Let $E \subseteq X$ be a maximal $\delta$-separated set with respect to the metric $d,$ and let $s (\delta) = |E|$ denote its finite cardinality. Let $\mu$ be a probability measure on $X$ supported entirely on $E$. Then recall the GNS triple $(L^2(\tau), \Lambda_\tau, [\delta_e 1])$.\\

Following the construction of Kerr, there exist unitaries $u_1, \dots, u_{s (\delta)} \in C(X)$ that are orthonormal with respect to the inner product induced by $\mu$ on $C(X)$, satisfying the \clipnorm\ bound
$$L_X(u_j) \leq C \delta^{-1},$$
for all $j \in \{1, \dots, s(\delta)\}$ and some constant $C > 1$. Furthermore, since $u_j$'s are unitaries, $\|u_j\|_{\infty} = 1.$ Let $\ell_{\Gamma}$ be the proper length function on $\Gamma$ induced from $\ell$ as in Lemma \ref{induced length}. For a fixed $k > p,$ define
$$I(\delta, k) : = \left\{t \in \Gamma : \ell_{\Gamma} (t) \leq (\delta^{-1})^{\frac{1}{k}} \right\}.$$
Let $N(\delta, k) = |I(\delta, k)|$. Since any length function on a finitely generated discrete group will at least have linear growth, it follows that there exists some constant $C' > 0$ such that $$N (\delta, k) \geq C' \left (\delta^{-1} \right )^{\frac {1} {k}}.$$ We define a finite subset $\mathcal{F} (\delta) \subseteq C_r^{\ast} (\mathcal G, \omega)$ as follows $:$
$$\mathcal{F} (\delta) = \left\{ \delta_t u_j : j = 1, \dots, s (\delta), \text{ and } t \in I(\delta, k) \right\}.$$
The cardinality of this set is precisely $|\mathcal{F} (\delta)| = s (\delta) \cdot N(\delta, k).$ 
Note that
$$L_{\ell}^{k}(\delta_t u_j) \leq \ell_{\Gamma} (t)^k \|u_j\|_{\infty} \leq \left( (\delta^{-1})^{\frac{1}{k}} \right)^k \cdot 1 = \delta^{-1}.$$
Taking the maximum of these components, we obtain $L^{\mathcal{K}_S, k}(\delta_g u_j) \leq \max(C\delta^{-1}, \delta^{-1}) \le C \delta^{-1}$. This shows that $C^{-1} \delta\ \mathcal F (\delta) \subseteq \mathcal L^{\mathcal K_S, k}_1.$
Let us consider the set of vectors $$\Omega (\delta) = \{ \Lambda_\tau(\delta_t u_j)[\delta_e 1]\ :\ t \in I (\delta, k),\ j = 1, \cdots, s (\delta) \} \subseteq L^2(\tau).$$ 
Therefore, by Lemma \ref{Voiculescu_orthonormal}, $\Omega(\delta)$ forms a finite orthonormal set of cardinality $s (\delta) \cdot N(\delta, k)$ in $L^2(\tau).$
By the virtue of Voiculescu's lemma (Theorem \ref{Voiculescu}) we have
$$D \left (\Omega (\delta), \frac {1} {2} \right ) \geq \frac {3} {4} \cdot s (\delta) \cdot N(\delta, k) \geq \frac {3} {4} \cdot s (\delta) \cdot C' \left (\delta^{-1} \right )^{\frac {1} {k}}.$$

\Bea
\frac {\log D \left (\mathcal L^{\mathcal K_S, k}_1, 2^{-1} C^{-1} \delta \right )} {\log 2 C \delta^{-1}} & \geq & \frac {\log D \left (C^{-1} \delta\ \mathcal F (\delta), 2^{-1} C^{-1} \delta \right )} {\log 2 C \delta^{-1}} \\ & = & \frac {\log D \left (\mathcal F (\delta), 2^{-1} \right )} {\log 2 C \delta^{-1}} \\ & \geq & \frac {\log D \left (\Omega (\delta), 2^{-1} \right )} {\log 2 C \delta^{-1}} \\ & \geq & \frac {\log \left (\frac {3} {4} \cdot s (\delta) \cdot C' \left (\delta^{-1} \right )^{\frac {1} {k}} \right )} {\log 2 C \delta^{-1}}.
\Eea
Taking limsup on both the sides we have
\Bea 
\mathrm {Mdim}_{L^{\mathcal K_S, k}} \left (C_r^{\ast} (\mathcal G, \omega) \right ) & \geq & \limsup\limits_{\delta \to 0^{+}} \frac {\log s (\delta)} {\log 2 C \delta^{-1}} + \frac {1} {k} \limsup\limits_{\delta \to 0^{+}} \frac {\log \delta^{-1}} {\log 2 C \delta^{-1}} \\ & = & \limsup\limits_{\delta \to 0^{+}} \frac {\log s (\delta)} {\log \delta^{-1}} \cdot \lim\limits_{\delta \to 0^{+}} \frac {\log 2 C \delta^{-1}} {\log \delta^{-1}} + \frac {1} {k} \cdot 1 \\ & = & \dim_{B} (X) + \frac {1} {k},
\Eea
as required.

\end{proof}
Now if we take $X$ to be a single point $\{\ast\}$, then the Kolmogorov dimension of $X$ is zero. We take the trivial action of a group of polynomial growth $\Gamma$ with respect to some length function $\ell_{\Gamma}$. Any $2$-cocycle $\omega$ on the corresponding groupoid becomes a $2$-cocycle on the group $\Gamma$. Also the length $\ell_{\Gamma}$ on $\Gamma$ can be viewed canonically as a length function on the groupoid and such length function tautologically satisfies Condition \ref{Lip-length}. Then $C^{\ast}_{r}(\Gamma\rtimes\{{\ast}\},\omega)\cong C_{r}^{\ast}(\Gamma,\omega)$ canonically. The metric stratification becomes a tautological one and consequently the stratified \clipnorm\ on $C^{\ast}_{r}(\Gamma,\omega)$ coincides with the \clipnorm\ $L^{k}_{\ell}$ as considered in \cite{chattopadhyay2026metric}. Therefore, we get the following corollary by combining Theorem \ref{general cocycle bound} and Theorem \ref{Mdim lower bound} which recovers Theorem 3.1 of \cite{chattopadhyay2026metric} and extends it for the $2$-cocycle twist case.
\begin{cor}
 Let $\Gamma$ be a finitely generated discrete group of polynomial growth of degree $r\geq 1$ with respect to a length function $\ell$. Let $\omega$ be a $2$-cocycle on $\Gamma$. Then for any $k>\frac{r}{2}$, 
 \begin{displaymath}
     \frac{1}{k}\leq\mathrm{Mdim}_{L_{\ell}^{k}}(C^{\ast}_{r}(\Gamma,\omega))\leq \frac{2r}{2k-r}.
 \end{displaymath}
\end{cor}

We end this subsection by noting an observation regarding the adjoint invariance of the stratified \clipnorm\ .
\begin{lem}
Let $\mathcal{G} : = \Gamma \rtimes X$ be the transformation groupoid associated with an {\bf isometric} action of $\Gamma$ on a compact metric space $(X, d)$ equipped with a normalized $2$-cocycle $\sigma : \Gamma \times \Gamma \rightarrow \mathbb T,$ viewed as a $2$-cocycle $\omega : \Gamma \times \Gamma \times X \rightarrow \mathbb T$ in the trivial manner, i.e., $\omega (t, u, x) = \sigma (t, u)$ for all $t, u \in \Gamma$ and $x \in X.$ Let $\ell$ be a continuous proper length function on $\mathcal G.$ Then for all $f = \sum\limits_{t \in \Gamma} \delta_t a_t \in C_c \left (\Gamma, \mathrm {Lip} (X) \right )$ and for all $k \in \mathbb N,$ we have $L^{\mathcal K_S, k} (f) = L^{\mathcal K_S, k} (f^{\ast_{\omega}}).$ 
\end{lem}

\begin{proof}
Fix some $k \in \mathbb N$ and $f = \sum\limits_{t \in \Gamma} \delta_t a_t \in C_c (\Gamma, \mathrm {Lip} (X))$ arbitrarily. Note that $$f^{\ast_{\omega}} = \sum\limits_{t \in \Gamma} \delta_{t^{-1}}\ \alpha_{t} \left (\overline {a_t} \right ) \overline {\sigma \left (t^{-1}, t \right )}.$$ Since $L_{\ell}^{k}$ is already adjoint invariant, by the definition of static stratified \clipnorm, it follows that 
\Bea
L^{\mathcal K_S, k} \left (f^{\ast_{\omega}} \right ) & = & \max \left \{L_{\ell}^{k} \left (f^{\ast_{\omega}} \right ), \sup\limits_{t \in \Gamma} L \left (\alpha_t \left (\overline {a_t} \right ) \overline {\sigma \left (t^{-1}, t \right )} \right ) \right \} \\ & = & \max \left \{L_{\ell}^{k} (f), \sup\limits_{t \in \Gamma} L \left (\alpha_t (\overline {a_t})  \right ) \right \} \\ & = & \max \left \{L_{\ell}^{k} (f), \sup\limits_{t \in \Gamma} L \left (\overline {a_t}  \right ) \right \} \quad (\because \mathrm {the\ action\ is\ isometric}) \\ & = & \max \left \{L_{\ell}^{k} (f), \sup\limits_{t \in \Gamma} L \left (a_t \right ) \right \} \\ & = & L^{\mathcal K_S, k} (f),
\Eea
as required.
\end{proof}
\subsubsection{Dynamic stratification} Bowen metrics were introduced by R. Bowen \cite{bowen1971entropy}*{Section 1} in order to define topological entropy of a topological dynamical system $(X, d, T),$ where $T$ is a uniformly continuous map on the metric space $(X, d).$ Motivated by this, we define dynamical stratification in terms of Bowen type metrics as follows $:$

\vspace{2mm}

Let $\ell_{\Gamma}$ be the induced proper length function on $\Gamma$ from the length function $\ell$ on $\mathcal{G}$. Let $E_{n}:=\{t\in\Gamma: \ell_{\Gamma}(t)\leq n\}$ for all $n\in\mathbb{N}\cup\{0\}$. Then clearly $E_{n}\subseteq E_{n+1}$ for all $n$. For each $n$, we define a Bowen type metric $d_{n}$ on $X$ given by $d_{n}(x,y):=\sup\limits_{t\in E_{n}}d(t\cdot x,t\cdot y)$. We define $d_{t}$ on $\{t\}\times X$ as the metric $d_{n}$ such that $t\in E_{n}$ for the smallest $n$. Clearly $d_{t}$ metrizes $X$ and $d_{t}\geq d_{e}=d$ for all $t\in\Gamma$. Then the choice $\{d_{t}\}_{t\in\Gamma}$ gives a metric stratification. We call it dynamic metric stratification and denote it by $\mathcal{K}_{D}$. The corresponding stratified \clipnorm\ will be denoted as usual by $L^{\mathcal{K}_{D},k}$. We start by noting that an isometric action of $\Gamma$ on $(X,d)$ reduces the dynamical stratification to the static one and therefore we get the analogue of the upper bound of the metric dimension for the dynamic stratification if the action of $\Gamma$ on $X$ is uniformly Lipschitz. The following lemma plays a crucial role in proving the analogues.

\begin{rem}
Note that if the discrete group of integers $\mathbb Z$ is acting on a compact metric space $(X, d)$ and is equipped with the word length function, then the associated Bowen type metric takes the following form $:$
$$d_n (x, y) = \sup\limits_{\left \lvert j \right \rvert \leq n} d (j \cdot x, j \cdot y),$$ for any pair of points $x, y \in X$ and for any $n \in \mathbb N \cup \{0\}.$
\end{rem}

\begin{lem}\label{uniformly Lipschitz}
Let $\mathcal{G} : = \Gamma \rtimes X$ be the transformation groupoid associated with an action of a countable discrete group $\Gamma$ on a compact metric space $(X, d)$ equipped with a normalized $2$-cocycle $\omega : \Gamma \times \Gamma \times X \rightarrow \mathbb T.$ Assume that $d$ is bi-Lipschitz equivalent to a metric $d'$ on $X.$ Let $\ell$ be a continuous proper length function on $\mathcal G,$ possessing the property of rapid decay with decay exponent $p > 1.$ If $k > p,$ then
$$\mathrm{Mdim}_{L_d^{\mathcal K_D,k}} (C_r^*(\mathcal{G}, \omega)) = \mathrm{Mdim}_{L_{d'}^{\mathcal K_D, k}} \left (C_r^{\ast} (\mathcal G, \omega) \right ).$$
\end{lem}

\begin{proof}
Let $\mathcal{L}_d$ denote the unit Lip-ball of $C_r^*(\mathcal{G}, \omega)$ with respect to the dynamic stratified \clipnorm\ $L_d^{\mathcal K_D, k}$ induced by the base metric $d$ and its associated Bowen metrics. Similarly consider $\mathcal L_{d'}$ and $L^{\mathcal K_D, k}_{d'}$ induced by the metric $d'.$

\vspace{2mm}

Since $d$ and $d'$ are uniformly bi-Lipschitz equivalent, there exist constants $c \in (0, 1]$ and $C \geq 1$ such that $c d(x, y) \leq d'(x, y) \leq C d (x,y)$ for all $x, y \in X.$ Exploiting the inequality on the left we have
$$d_n(x, y) = \max_{t \in E_n} d(t \cdot x, t \cdot y) \le \max_{t \in E_n} \frac{1}{c} d'(t \cdot x, t \cdot y) = \frac{1}{c} d_n'(x, y),$$ 
Let $f = \sum\limits_{t \in \Gamma} \delta_t a_t \in \mathcal{L}_d.$ Fix some $t_0 \in \Gamma$ arbitrarily and get hold of the smallest positive integer $n_0$ such that $t_0 \in E_{n_0}.$ Let $L_{t_0, d}$ and $L_{t_0, d'}$ be the \clipnorms\ on $C (X)$ associated to the Bowen metrics $d_n$ and $d_n'$ respectively. By definition, $L_{t_0, d} (a_{t_0}) \leq 1.$ Then the above metric inequality strictly forces $L_{t_0, d'}(a_t) \leq \frac{1}{c}.$ Since the condition $L_{\ell}^k \left (\sum\limits_{t \in \Gamma} \delta_t a_t \right ) \leq 1$ is completely independent of the choice of base metric, it follows that $L_{d'}^{\mathcal K_D, k}(f) \le \frac{1}{c}.$ 
This shows that $\mathcal{L}_d \subseteq \frac{1}{c} \mathcal{L}_{d'}.$ Thus by linear scaling it follows that
$$D(\mathcal{L}_d, \delta) \le D\left(\frac{1}{c} \mathcal{L}_{d'}, \delta\right) = D(\mathcal{L}_{d'}, c\delta).$$
Taking logarithm in both the sides and evaluating the limit supremum as $\delta \to 0^+,$ we have
\Bea
\limsup_{\delta \to 0^+} \frac{\log D(\mathcal{L}_d, \delta)}{\log \delta^{-1}} & \leq & \limsup_{\delta \to 0^+} \frac{\log D(\mathcal{L}_{d'}, c\delta)}{\log \delta^{-1}} \\
& = & \limsup_{\delta \to 0^+} \left[ \frac{\log D(\mathcal{L}_{d'}, c\delta)}{\log(c\delta)^{-1}} \cdot \frac{\log(c\delta)^{-1}}{\log \delta^{-1}} \right] \\
& = & \mathrm{Mdim}_{L_{d'}^{\mathcal K_D, k}} (C_r^*(\mathcal{G}, \omega)).
\Eea
This shows that 
$$\mathrm{Mdim}_{L_d^{\mathcal K_D,k}} (C_r^*(\mathcal{G}, \omega)) \leq \mathrm{Mdim}_{L_{d'}^{\mathcal K_D, k}} \left (C_r^{\ast} (\mathcal G, \omega) \right ).$$ Similarly using the inequality $d' (x, y) \leq C d (x y)$ for all $x, y \in X,$ we have
$$\mathrm{Mdim}_{L_{d'}^{\mathcal K_D, k}} \left (C_r^{\ast} (\mathcal G, \omega) \right ) \leq \mathrm{Mdim}_{L_d^{\mathcal K_D,k}} (C_r^{\ast}(\mathcal{G}, \omega)).$$
This completes the proof.

\end{proof}

\begin{thm}\label{dynamic_finite_bound}
Let $\mathcal{G}$ be the transformation groupoid associated with an action of $\Gamma$ on a compact metric space $(X, d)$ equipped with a normalized $2$-cocycle $\omega : \Gamma \times \Gamma \times X \rightarrow \mathbb T.$ Let the action of $\Gamma$ on $X$ be uniformly Lipschitz. Let $\ell$ be a continuous proper length function on $\mathcal G$ satisfying the condition \eqref{Lip-length}, under which $\mathcal G$ possesses the property of polynomial growth of degree $r \geq 1.$ Let $C_r^{\ast}(\mathcal {G}, \omega)$ be equipped with the dynamic stratified \clipnorm\ $L^{\mathcal{K}_{D}, k}.$ Then for all $k > \frac{r}{2}$, we have
\begin{displaymath}\mathrm{Mdim}_{L^{\mathcal{K}_{D},k}} (C_r^{\ast}(\mathcal{G}, \omega)) \leq \frac {2 r} {2 k - r} + M_X \cdot \left (\frac {2 k + r} {2 k - r} \right ).\end{displaymath}
\end{thm}
\begin{proof}
    As the action is uniformly Lipschitz, we can choose a metric $d^{\prime}$ on $X$ which is strongly equivalent to $d$ such that the action is isometric with respect to $d^{\prime}$. In that case the Kolmogorov dimension of $X$ with respect $d'$ remains unaltered. Now note that for an isometric action, the dynamic stratification coincides with the static one. If we denote the corresponding stratified \clipnorm\ by $L^{\mathcal{K}_{D},k}_{d^{\prime}},$ then by Theorem \ref{general cocycle bound}, 
    \begin{displaymath}
     \mathrm{Mdim}_{L^{\mathcal{K}_{D},k}_{d^{\prime}}} (C_r^{\ast}(\mathcal{G}, \omega)) \leq \frac {2 r} {2 k - r} + M_X \cdot \left (\frac {2 k + r} {2 k - r} \right ).   
    \end{displaymath}
    But by Lemma \ref{uniformly Lipschitz}, 
    \begin{displaymath}
        \mathrm{Mdim}_{L^{\mathcal{K}_{D},k}_{d^{\prime}}} (C_r^{\ast}(\mathcal{G}, \omega))=\mathrm{Mdim}_{L^{\mathcal{K}_{D},k}} (C_r^{\ast}(\mathcal{G}, \omega)).
    \end{displaymath}
    Therefore, the result follows.
\end{proof}
Using the bound obtained in Theorem \ref{betterbound}, the same argument used in the last theorem produces the following theorem: 
\begin{thm}\label{dynamicbound}
Let $\mathcal{G}$ be the transformation groupoid associated with an action of $\Gamma$ on a compact metric space $(X, d)$ equipped with a normalized $2$-cocycle $\sigma : \Gamma \times \Gamma \rightarrow \mathbb T.$ Let the action of $\Gamma$ on $X$ be uniformly Lipschitz. Let $\ell$ be a continuous proper length function on $\mathcal G$ satisfying the condition \eqref{Lip-length}, under which $\mathcal G$ possesses the property of polynomial growth of degree $r \geq 1.$ Let $C_r^{\ast}(\mathcal {G}, \sigma)$ be equipped with the dynamic stratified \clipnorm\ $L^{\mathcal{K}_{D}, k}.$ Then for all $k > \frac{r}{2}$, we have
\begin{displaymath}\mathrm{Mdim}_{L^{\mathcal{K}_{D},k}} (C_r^{\ast}(\mathcal{G}, \sigma)) \leq  \frac {4 k + r} {2 k - r} \left( M_{\Gamma} + M_X \right). \end{displaymath}
\end{thm}

Now we shall show that the assumption of uniformly Lipschitz action in the Theorem \ref{dynamic_finite_bound} and Theorem \ref{dynamicbound} is really important. To be more precise, we shall produce an example where a group of polynomial growth acts on a compact metric space of finite Kolmogorov dimension, the action is uniformly equicontinuous and yet the metric dimension for the CQMS structure   of the associated transformation groupoid $C^{\ast}$-algebra coming from the dynamic stratified \clipnorm\ is $+\infty$. To that end, we first obtain a lower bound for the metric dimension for a general action.

\begin{thm}
Let $\mathcal{G}$ be the transformation groupoid associated with an action of $\Gamma$ on a compact metric space $(X, d)$ equipped with a normalized $2$-cocycle $\omega : \Gamma \times \Gamma \times X \rightarrow \mathbb T.$ Let $\ell$ be a continuous proper length function on $\mathcal G$ satisfying the condition \eqref{Lip-length}, under which $\mathcal G$ possesses the property of rapid decay with decay exponent $p > 0.$ Let $C_r^{\ast}(\mathcal {G}, \omega)$ be equipped with the dynamic stratified \clipnorm\ $L^{\mathcal{K}_{D}, k}.$ Then for all $k > p,$ we have
$$\mathrm {Mdim}_{L^{\mathcal K_D, k}} \left (C_r^{\ast} (\mathcal G, \omega) \right) \geq \sup\limits_{n \geq 1} \dim_B (X, d_n) + \frac{1}{k}.$$
\end{thm}

\begin{proof}
The proof is essentially a modification of the proof of Theorem \ref{Mdim lower bound}. Fix $n \ge 1$ and a $\delta > 0$. Let $E_n \subseteq X$ be a maximal $\delta$-separated set with respect to the Bowen metric $d_n$, and let $\left \lvert E_n \right \rvert : = s (\delta) < \infty.$ Let $\mu$ be a probability measure on $X$ supported on $E_n.$ Again recall the GNS triple $(L^{2}(\tau),\Lambda_{\tau},[\delta_{e}1])$.\\

Following the construction of Kerr, there exist unitaries $u_1, \dots, u_{s (\delta)} \in C(X)$ that are orthonormal with respect to the inner product induced by $\mu$ on $C(X)$, satisfying the Lip-norm bound
$$L_n(u_j) \leq C \delta^{-1},$$
for all $j \in \{1, \dots, s(\delta)\}$ and some constant $C > 1$. Furthermore, since they are unitaries, $\|u_j\|_{\infty} = 1$. Let $\ell_{\Gamma}$ be the proper length function on $\Gamma$ induced from $\ell$ as in Lemma \ref{induced length}. For a fixed $k > p,$ define
$$I(\delta, k, n) : = \left\{t \in \Gamma : t \notin B_{\ell_{\Gamma}} (n) \text{ and } \ell_{\Gamma} (t) \leq (\delta^{-1})^{\frac{1}{k}} \right\}.$$
Let $N(\delta, k, n) = |I(\delta, k, n)|$. Since any length function on a finitely generated discrete group will at least have linear growth, it follows that there exists some constant $C' > 0$ such that $$N (\delta, k, n) \geq C' \left (\delta^{-1} \right )^{\frac {1} {k}} - \left \lvert B_{\ell_{\Gamma}} (n) \right \rvert.$$ We define a finite subset $\mathcal{F} (\delta) \subseteq C_r^{\ast} (\mathcal G, \omega)$ as follows $:$
$$\mathcal{F} (\delta) = \left\{ \delta_t u_j : j = 1, \dots, s (\delta), \text{ and } t \in I(\delta, k, n) \right\}.$$
The cardinality of this set is precisely $|\mathcal{F} (\delta)| = s (\delta) \cdot N(\delta, k, n)$. 

We evaluate the stratified \clipnorm\ for an arbitrary element $\delta_t u_j \in \mathcal{F} (\delta).$ Let $m$ be the least positive integer such that $t \in B_{\ell_{\Gamma}} (m).$ Since $t \notin B_{\ell_{\Gamma}} (n),$ it follows that $m > n.$ Thus by the monotonicity of the Bowen metrics we have $d_m (x, y) \geq d_n(x, y)$ for all $x, y \in X$. Therefore,
$$L_t (u_j) = L_m (u_j) \leq L_n(u_j) \leq C \delta^{-1}.$$
Also,
$$L_{\ell}^{k}(\delta_t u_j) \leq \ell_{\Gamma} (t)^k \|u_j\|_{\infty} \leq \left( (\delta^{-1})^{\frac{1}{k}} \right)^k \cdot 1 = \delta^{-1}.$$
Taking the maximum of these components, we obtain $L^{\mathcal{K_D}, k}(\delta_t u_j) \leq \max(C\delta^{-1}, \delta^{-1}) \le C \delta^{-1}$. This shows that $C^{-1} \delta\ \mathcal F (\delta) \subseteq \mathcal L^{\mathcal K_D, k}_1.$
Let us consider the set of vectors $$\Omega (\delta) : = \{ \Lambda_\tau(\delta_t u_j)[\delta_e 1]\ :\ t \in I (\delta, k, n),\ j = 1, \cdots, s (\delta) \} \subseteq L^2(\tau).$$ Then again by Lemma \ref{Voiculescu_orthonormal}, $\Omega(\delta)$ is a finite  orthonormal set of cardinality $s (\delta) \cdot N(\delta, k, n)$ in $L^2(\tau).$

Let $D(\mathcal L^{\mathcal{K_D}, k}_1, \delta)$ denote the infimum of the dimensions of finite-dimensional subspaces of $L^2(\tau)$ required to approximate the unit ball $\mathcal L^{\mathcal{K}_D, k}_1$ within $\delta$. By the virtue of Voiculescu's lemma, we have
$$D \left (\Omega (\delta), \frac {1} {2} \right ) \geq \frac {3} {4} \cdot s (\delta) \cdot N(\delta, k, n) \geq \frac {3} {4} \cdot s (\delta) \cdot \left (C' \left (\delta^{-1} \right )^{\frac {1} {k}} - \left \lvert B_{\ell_{\Gamma}} (n) \right \rvert \right ).$$
Since $n \geq 1$ is a fixed natural number, independent of $\delta > 0,$ we can choose $\delta > 0$ sufficiently small such that $$C' \left (\delta^{-1} \right )^{\frac {1} {k}} - \left \lvert B_{\ell_{\Gamma}} (n) \right \rvert \geq \frac {C'} {2} \left (\delta^{-1} \right )^{\frac {1} {k}}.$$ For all such $\delta > 0$ we have
\Bea
\frac {\log D \left (\mathcal L^{\mathcal K_D, k}_1, 2^{-1} C^{-1} \delta \right )} {\log 2 C \delta^{-1}} & \geq & \frac {\log D \left (C^{-1} \delta\ \mathcal F (\delta), 2^{-1} C^{-1} \delta \right )} {\log 2 C \delta^{-1}} \\ & = & \frac {\log D \left (\mathcal F (\delta), 2^{-1} \right )} {\log 2 C \delta^{-1}} \\ & \geq & \frac {\log D \left (\Omega (\delta), 2^{-1} \right )} {\log 2 C \delta^{-1}} \\ & \geq & \frac {\log \left (\frac {3} {4} \cdot s (\delta) \cdot \frac {C'} {2} \left (\delta^{-1} \right )^{\frac {1} {k}} \right )} {\log 2 C \delta^{-1}}
\Eea
Taking limsup on both the sides of the above inequality we have
\Bea 
\mathrm {Mdim}_{L^{\mathcal K_D, k}} \left (C_r^{\ast} (\mathcal G, \omega) \right ) & \geq & \limsup\limits_{\delta \to 0^{+}} \frac {\log s (\delta)} {\log 2 C \delta^{-1}} + \frac {1} {k} \limsup\limits_{\delta \to 0^{+}} \frac {\log \delta^{-1}} {\log 2 C \delta^{-1}} \\ & = & \limsup\limits_{\delta \to 0^{+}} \frac {\log s (\delta)} {\log \delta^{-1}} \cdot \lim\limits_{\delta \to 0^{+}} \frac {\log 2 C \delta^{-1}} {\log \delta^{-1}} + \frac {1} {k} \cdot 1 \\ & = & \dim_{B} \left (X, d_n \right ) + \frac {1} {k}.
\Eea
Since this holds for all $n \geq 1,$ taking supremum over all $n \geq 1,$ the desired result follows.

\end{proof}

Now we shall produce the example. To that end, for each integer $m \ge 1$, let $Y_m = \{0, 1\}^m$ be equipped with the discrete topology. We write elements of $Y_m$ as $u = (u_0, \dots, u_{m-1})$. Define the product space
$$X = \prod_{m \ge 1} Y_m,$$ equipped with the product topology inherited from the discrete topology on $Y_m.$
Clearly $X$ is compact. For a fixed $m \ge 1$, we define a local distance function $\rho_m$ on $Y_m$ by
$$\rho_m(u, v) = \max_{0 \le j < m} \eta_m^j 1_{u_j \neq v_j},$$
where $\eta_m = 2^{-2^m}$ and $1_{u_j \neq v_j}$ is the indicator function. We define the global distance function $d : X \times X \to [0, \infty)$ by
$$d(x, y) = \sup_{m \ge 1} 2^{-m} \rho_m(x_m, y_m).$$
Let $T : X \rightarrow X$ be defined as $$T : = \prod\limits_{m \geq 1} S_m,$$ where $S_m : Y_m \rightarrow Y_m$ is the cyclic permutation given by
$$S_m \left (u_0, u_1, \cdots, u_{m - 1} \right ) = \left (u_1, \cdots, u_{m - 1}, u_{0} \right ).$$
We call this dynamical system $(X,d,T)$ the {\bf symbolic dynamical system}.
Now for each $n \geq 1,$ define the Bowen metric $d_n : X \times X \rightarrow [0, \infty)$ by
$$d_n (x, y) = \max\limits_{-n \leq k \leq n} d \left (T^{k} x, T^{k} y \right ),$$ $x,y \in X.$ Note that this metric coincides with the Bowen type metric used to define the dynamic stratified \clipnorm\ if the discrete group of integers is acting on the base space $X$ and is endowed with the word length function.  

\begin{thm}
For any $n \ge 1$, the Kolmogorov dimension of $X$ with respect to the Bowen metric $d_n$ is $\dim_B(X, d_n) = 2 n + 1.$
\end{thm}

\begin{proof}
Fix $n \ge 1$ and Let $\varepsilon \in (0, 1)$ be arbitrary. We can get hold of a unique integer $M \ge 0$ such that
$$2^{-(M+1)} < \varepsilon \le 2^{-M}.$$
Then $\varepsilon \to 0^{+}$ strictly forces $M \to \infty.$ Taking logarithm in both the sides we have
$$M \log 2 \le -\log \varepsilon < (M + 1) \log 2.$$
By the monotonicity of the covering number with respect to radius, we have
$$\log N(X, d_n, 2^{-M}) \le \log N(X, d_n, \varepsilon) \le \log N(X, d_n, 2^{-(M+1)}).$$
We now evaluate the covering number $N(X, d_n, 2^{-M})$ for sufficiently large $M$ such that $\left \lfloor \log_2 M \right \rfloor > 2 n.$
By definition of $d,$ for all $x, y \in X,$ we have
$$d_n(x, y) = \max_{- n \leq k \leq n} \sup_{m \ge 1} \max_{0 \le j < m} 2^{-m} \eta_m^{(j - k) \bmod m} 1_{x_{m, j} \neq y_{m, j}},$$
where $\eta_m = 2^{-2^m}$. Two points will lie in different Bowen balls of radius $2^{-M}$ if they agree on some coordinate $(m, j)$ which would contribute a $d_n$-distance of $2^{-M}$ or more. Henceforth these coordinates will be called visible coordinates.
Let $K_n(M)$ be the total number of such visible coordinates across all blocks. Since each coordinate is taking the values $0$ or $1,$ it follows that $N(X, d, 2^{-M}) = 2^{K_n(M)}.$ A coordinate $j$ in block $m$ is visible if there exists at least one $k \in \{0, \dots, n-1\}$ such that
$$2^{-m} \left(2^{-2^m}\right)^{(j - k) \bmod m} \ge 2^{-M} \iff m + 2^m \big((j - k) \bmod m\big) \le M.$$
Let $K_{m,n}$ be the number of visible coordinates in the $m$-th block $Y_m.$ We evaluate $K_{m,n}$ by partitioning the blocks into the following three regimes $:$

\begin{enumerate}
\item $1 \leq m \leq n:$ For any coordinate $j \in \{0, \dots, m-1\}$, we can choose $k = j.$ Since $m < n$, this $k$ is lying within $\{0, \dots, n-1\}.$ This choice forces $(j - k) \bmod m = 0.$ So the visibility condition reduces to $m \leq M,$ which is unconditionally true since $m < n < \left \lfloor \log_2 M \right \rfloor < M$. Thus, in this case every coordinate is visible and hence $$K_{m,n} = m.$$

\item $n < m \le \lfloor \log_2 M \rfloor:$ In this case the inequality involving visible coordinates may hold for multiple coordinates, but we can trivially bound it by the total size of the block, i.e., $K_{m,n} \le m$.

\item $\lfloor \log_2 M \rfloor < m \leq M:$
Here, $2^m > M \ge M - m$. For the visibility condition $2^m \big((j - k) \bmod m\big) \le M - m$ to hold, the multiplier must strictly be zero and hence $$(j - k) \bmod m = 0 \implies j \equiv k \pmod m.$$ Now since $0 \leq j < m,$ for $0 \leq k \leq n < m$, the congruence implies exact equality, i.e., $j = k$ and for $- n \leq k < 0,$ the congruence implies $j = m + k.$ Since $k$ ranges over $\{- n, \dots, n\},$ the set of visible coordinates in block $m$ is given by $$S : = \{0, \cdots, n \} \cup \{m - n, \cdots, m - 1 \}.$$
Since $\left \lfloor \log_2 M \right \rfloor > 2 n,$ it follows that $m - n > n.$ Thus all the elements of $S$ are dstinct and hence $K_{m, n} = \left \lvert S \right \rvert = 2 n + 1.$ 
    
\item $m > M :$ The condition $m \le M$ fails even if $j=k$ mod $m,$ so those blocks contribute nothing.  
\end{enumerate}

\vspace{2mm}

Summing up the visible coordinates across all blocks yields $K_n(M).$ Thus
\Bea
K_n(M) & = & \sum_{m=1}^{n} m + \sum_{m = n + 1}^{\lfloor \log_2 M \rfloor} K_{m,n} + \sum_{m=\lfloor \log_2 M \rfloor + 1}^M (2 n + 1) \\ & \leq & \sum\limits_{m = 1}^{\lfloor \log_2 M \rfloor} m + (2 n + 1) \left (M - \lfloor \log_2 M \rfloor \right ) \\ & = & \frac {\lfloor \log_2 M \rfloor \left (\lfloor \log_2 M \rfloor + 1 \right )} {2} + (2 n + 1) \left (M - \left \lfloor \log_2 M \right \rfloor \right ) \\ & = & (2 n + 1) M + O \left ((\log_2 M)^2 \right ).
\Eea
Also since each $K_{m, n} \geq 0,$ taking the minimal contribution over all blocks of size $1 \leq m \leq \left \lfloor \log_2 M \right \rfloor,$ it follows that $$K_n (M) \geq (2 n + 1) \left (M - \left \lfloor \log_2 M \right \rfloor \right ) = (2 n + 1) M + O \left (\log_2 M \right ).$$
Thus we have
$$\frac{(2 n + 1) M + O \left (\log_2 M \right )}{M+1} \le \frac{\log N(X, d_n, \varepsilon)}{-\log \varepsilon} \le \frac{(2 n + 1) (M+1) + O((\log_2(M+1))^2)}{M}.$$
Finally, we take the limit inferior and limit superior as $\varepsilon \to 0^+$ (which forces $M \to \infty$)
\Bea
2 n + 1 & = & \lim_{M \to \infty} \frac{(2 n + 1) M + O \left (\log_2 M \right )}{M+1} \\ & \leq & \liminf_{\varepsilon \to 0^+} \frac{\log N(X, d_n, \varepsilon)}{-\log \varepsilon} \\
& \leq & \limsup_{\varepsilon \to 0^+} \frac{\log N(X, d_n, \varepsilon)}{-\log \varepsilon} \\ & \leq & \lim_{M \to \infty} \frac{(2 n + 1) (M+1) + O \left (\left (\log_2 (M + 1) \right )^{2} \right )}{M} \\ & = & 2 n + 1.
\Eea
By the Squeeze Theorem, the limits coincide. Therefore, $\dim_B(X, d_n) = 2 n + 1.$ In particular, the Kolmogorov dimension of $X$ with respect to the base metric is $1$.
\end{proof}

\begin{thm}
The dynamical system $(X, d, T)$ is uniformly equicontinuous. That is, for every $\varepsilon > 0$, there exists a $\delta > 0$ such that for all $k \in \mathbb Z$ and for any pair of points $x, y \in X$ with $d(x, y) < \delta$, we have $d(T^k x, T^k y) < \varepsilon$. In other words, the family of maps $\left \{T^{k} \right \}_{k \in \mathbb Z}$ is uniformly equicontinuous on $X.$
\end{thm}

\begin{proof}
Let $\varepsilon > 0$ be given. Choose an integer $M \ge 1$ sufficiently large such that $2^{-M} < \varepsilon.$
Note that for any $m \geq 1,$ the maximum possible value of $\rho_m$ on $Y_m$ is bounded above by $1.$ Therefore,
$$\sup_{m > M} 2^{-m} \rho_m(S_m^k x_m, S_m^k y_m) \le \sup_{m > M} 2^{-m} = 2^{-(M+1)} < 2^{-M} < \varepsilon.$$
Also note the minimum possible non-zero distance within block $m$ occurs when two points disagree only at the coordinate bearing the smallest weight, $j = m-1,$ which is given by $2^{-m} \eta_m^{m-1} = 2^{-m} 2^{-2^m(m-1)}.$
Define $\delta > 0$ in such a way that
$$ \delta = \frac{1}{2} \min_{1 \leq m \leq M} \left( 2^{-m} 2^{-2^m(m-1)} \right) > 0.$$
Now suppose that $x, y \in X$ are any two points satisfying $d(x, y) < \delta.$
If $x_m \neq y_m$ for any block $m \leq M,$ then $d(x, y)$ would be bounded below by the minimum non-zero distance for that block, which is strictly greater than $\delta.$ This contradicts our assumption that $d(x, y) < \delta.$
Therefore, it follows that $x_m = y_m$ for all $1 \leq m \leq M.$ Thus, for any $k \in \mathbb{Z}$
$$S_m^k x_m = S_m^k y_m \implies \rho_m(S_m^k x_m, S_m^k y_m) = 0 \quad \text{for all } 1 \le m \le M.$$
Then for any $k \in \mathbb Z$ and for any pair of points $x, y \in X$ with $d (x, y) < \delta$ we have
\Bea
d(T^k x, T^k y) & = & \max \left \{ \sup_{1 \le m \le M} 2^{-m}(0), \sup_{m > M} 2^{-m} \rho_m (S_m^k x_m, S_m^k y_m) \right \} \\
& \leq & \sup_{m > M} 2^{-m} \rho_m (S_m^{k} x_m, S_m^{k} y_m) \\
& < & \varepsilon.
\Eea
This shows that the family of maps $\{T^k\}_{k \in \mathbb{Z}}$ is uniformly equicontinuous on $X,$ as required.
\end{proof}
Now consider the transformation groupoid $\mathcal{G}=\mathbb{Z}\rtimes X$ associated with the symbolic $C^{\ast}$-dynamical system $(X,d,T)$. Consider the length function $\ell$ on $\mathcal{G}$ induced by the word length function $\ell_{\mathbb{Z}}$ on $\mathbb{Z}$. Then by Remark \ref{grouplengthtogroupoidlength}, $\mathcal{G}$ has the polynomial growth property with growth exponent $1$ with respect to $\ell$ such that $\ell$ satisfies \ref{Lip-length}.
\begin{cor}\label{dynamic_complexity}
Consider the transformation groupoid $\mathcal G = \mathbb Z \rtimes X$ associated to the symbolic $C^{\ast}$-dynamical system $(X, d, T).$ Let $\ell$ be the length function induced by $\ell_{\mathbb{Z}}$. Then for all $k \geq \left \lceil \frac{1}{2} \right \rceil = 1,$ we have 
$\mathrm {Mdim}_{L^{\mathcal K_D, k}} \left (C_r^{\ast} (\mathcal G, \omega) \right ) = \infty.$ Also by Theorem \ref{dynamic_finite_bound}, the action cannot be uniformly Lipschitz.
\end{cor}

\subsection{The case of exponential growth} In this subsection, we consider cocycle twists of transformation groupoids $\mathcal{G}=\Gamma \rtimes X,$ where $\Gamma$ is a countable discrete group having exponential growth. For further details regarding groups of exponential growth, the readers can refer to \cite{delaharpe}*{Chapter VII}. Fix some normalized $2$-cocycle $\omega$ on $\mathcal{G}$. Let $\ell$ be a length function on $\mathcal{G}$ such that the groupoid has twisted rapid decay property with respect to $\ell$. We would like to mention that even if the group $\Gamma$ has rapid decay property, the groupoid might fail to admit rapid decay property. This typically happens when the action is free. Then the corresponding groupoid is prinicipal étale. By \cite{Weygandt-2024-Rapid-decay-for-etale-gpd}*{Theorem 3.2}, the groupoid has rapid decay if and only if it has polynomial growth property. Therefore, if a group $\Gamma$ with exponential growth property has a free action, the resulting groupoid can not have rapid decay property. So to obtain rapid decay property one has to look at non-free actions of groups of exponential growth.\\
Consider the \clipnorm\ $L^{\mathcal{K},k}$ coming from some metric stratification where $k>p$, $p$ is the exponent of rapid decay. Recall the length function $\ell_{\Gamma}$ on $\Gamma$ induced from the length function $\ell$ on the groupoid.

\begin{rem}
For a transformation groupoid arising from the action of a discrete group with exponential growth on a compact Hausdorff space $X$, we always assume the underlying group to be countable unless otherwise specified.
\end{rem}

\begin{thm}\label{stratification_infinity}
Let $\mathcal G$ be the transformation groupoid associated with an action of a discrete group $\Gamma$ on a compact metric space $(X, d);$ $\omega$ be a normalized $2$-cocycle on $\mathcal{G};$ $\ell$ be a continuous proper length function on $\mathcal G$ with respect to which $\mathcal G$ possesses the property of rapid decay with decay exponent $p > 0.$ If $\Gamma$ has exponential growth with respect to the length function $\ell_{\Gamma}$, then for all $k > p,$ $\mathrm {Mdim}_{L^{\mathcal K, k}} \left (C_r^{\ast} (\mathcal G,\omega) \right ) = \infty.$
\end{thm}

\begin{proof}
Since $\Gamma$ has exponential growth with respect to $\ell_{\Gamma},$ there exists $M > 1$ and $N \in \mathbb N$ such that for all $n > N$ we have 
$$\left \lvert \left \{t \in \Gamma\ : \ell_{\Gamma} (t) \leq n \right \} \right \rvert \geq M^{n}.$$
Fix some $\delta > 0$ arbitrarily in such a way that $\delta < \frac {1} {N^{k}}.$ Let $$U_{\delta} : = \left \{\delta_t 1\ :\ \ell_{\Gamma} (t) \leq \left (\delta^{-1} \right )^{\frac {1} {k}} \right \} \subseteq C_c (\mathcal G, \omega)_1.$$ Then $$\left \lvert U_{\delta} \right \rvert \geq M^{\left (\delta^{-1} \right )^{\frac {1} {k}}}.$$ Also, since $$L^{\mathcal K, k} \left (\delta_{t} 1 \right ) = L_{\ell}^{k} \left (\delta_t 1 \right ) \leq \left \|M_{\ell}^{k} \left (\delta_{t} 1 \right ) \right \|_{I} = \sup\limits_{x \in X} \ell (t, x)^{k} \leq \ell_{\Gamma} (t)^{k},$$ for all $t \in \Gamma,$ it follows that $\delta U_{\delta} \subseteq \mathcal L_1^{\mathcal K, k}.$ 

Let $\mu$ be a probability measure on $X.$ Then recall the GNS triple $(L^{2}(\tau),\Lambda_{\tau},[\delta_{e}1])$.
Let us consider the set of vectors $$\Omega_{\delta} = \{ \Lambda_\tau \left (\delta_t 1 \right ) [\delta_e 1]\ :\ t \in U_{\delta} \} \subseteq L^2(\tau).$$ 
Again by Lemma \ref{Voiculescu_orthonormal}, $\Omega_{\delta}$ is a finite set of orthonormal vectors in $L^{2} (\tau).$ Therefore, again by the virtue of Voiculescu's lemma, we have 
$$D \left (\Omega_{\delta}, 2^{-1} \right ) \geq \frac {3} {4} \left \lvert \Omega_{\delta} \right \rvert = \frac {3} {4} \left \lvert U_{\delta} \right \rvert \geq \frac {3} {4} M^{\left ( \delta^{-1} \right )^{\frac {1} {k}}}.$$ Then we have
\Bea \mathrm {Mdim}_{L^{\mathcal K, k}} \left (C_{r}^{\ast} (\mathcal G, \omega) \right ) & = & \limsup\limits_{\delta \to 0^{+}} \frac {\log D \left (\mathcal L^{\mathcal K, k}_{1}, 2^{-1} \delta \right )} {\log 2 \delta^{-1}} \\ & \geq & \limsup_{\delta \to 0^{+}} \frac {\log D \left (\delta U_{\delta}, 2^{-1} \delta \right )} {\log 2 \delta^{-1}} \\ & = & \limsup_{\delta \to 0^{+}} \frac {\log D \left (U_{\delta}, 2^{-1} \right )} {\log 2 \delta^{-1}} \\ & \geq & \limsup\limits_{\delta \to 0^{+}} \frac {\log D \left (\Omega_{\delta}, 2^{-1} \right )} {\log 2 \delta^{-1}} \\ & \geq & \limsup\limits_{\delta \to 0^{+}} \frac {\log M^{\left (\delta^{-1} \right )^{\frac {1} {k}}}} {\log \delta^{-1}} \\ & = & \log M\ \limsup\limits_{\delta \to 0^{+}} \frac {\left (\delta^{-1} \right )^{\frac {1} {k}}} {\log \delta^{-1}} \\ & = & \infty. 
\Eea
This completes the proof.
\end{proof}
\begin{rem}
    As observed earlier, taking $X$ to be a single point $\{\ast\}$ and considering the trivial action of the group, one can recover Theorem 3.16 of \cite{joardar2025metrics} and extend it for the cocycle twisted group $C^{\ast}$-algebra.
\end{rem}

In fact, we shall show that any cocycle twisted groupoid $C^{\ast}$-algebra where the acting group has exponential growth, has the metric dimension $+\infty$ generically. We shall prove that by exploiting entropy-like quantity of automorphisms. Recall that an automorphism $\varphi$ of $\Gamma$ and a homeomorphism $\psi$ of $X$ combine to produce an automorphism of $C^{\ast}_{r}(\mathcal{G},\omega)$ provided for any $f \in C(X),$ $t, u \in \Gamma$ and $x \in X,$ we have
$$\hat{\psi}(\alpha_t (f)) = \alpha_{\varphi(t)}(\hat{\psi}(f)),$$ and,
$$\omega (\varphi (t), \varphi (u), x) = \omega (t, u, \psi (x)),$$
where $\hat{\psi}$ is the induced automorphism on $C(X)$. The automorphism $\Phi$ is given on the dense $\ast$-subalgebra $C_{c}(\mathcal{G},\omega)$ by
\begin{displaymath}
    \Phi(\sum\delta_{t}a_{t})=\sum\delta_{\phi(t)}\hat{\psi}(a_{t}).
\end{displaymath}

\vspace{2mm}

Let $\Phi$ be an automorphism of $C^{\ast}_{r}(\mathcal{G},\omega)$ induced from automorphism $\varphi$ of $\Gamma$ and a homeomorphism $\psi$ of $X$; $A$ be any subspace of $C^{\ast}_{r}(\mathcal{G},\omega)$. We denote the collection of all finite subsets of $A$ in the unit ball of $A$ by $Pf(A)_{1}$. For any $\Omega\in Pf(A_1),$ denote $\prod\limits_{i=1}^{n} \Phi^{i - 1}(\Omega)$ by $\Omega_{n}.$ For $\Omega\in Pf(A_{1})$ and $\delta>0$, define
    \begin{align*}
&&\mathrm{Ent}_{A} (\Phi,\Omega,\delta) & =  \limsup\limits_{n\to\infty}\frac{1}{n}\log D \left (\Omega_n,\delta \right ),\\
&&\mathrm{Ent}_{A}(\Phi,\Omega) & = \sup_{\delta>0}\mathrm{Ent}_{A}(\Phi,\Omega,\delta),\\
&&\mathrm{Ent}_{A}(\Phi) & =  \sup_{\Omega \in Pf(A_1)}\mathrm{Ent}_{A}(\Phi,\Omega).
\end{align*}
\begin{rem}
    Note that when $A=\mathrm{Dom}(L)$ for a CQMS $(B,L)$, our definition of $\mathrm{Ent}_{A}(\alpha)$ coincides with the definition of product entropy of an automorphism as given in \cite{Kerr}*{Definition 5.2}. In our notation the product entropy of an automorphism $\Phi$ will be $\mathrm{Ent}_{\mathrm{Dom}(L)}(\Phi)$. 
\end{rem}
We recall the definition of the algebraic entropy of an automorphism $\alpha$ of a discrete group $\Gamma.$
\begin{defn}\label{algentp}
The algebraic entropy of an automorphism $\phi$ of a discrete group $\Gamma$ is defined to be
\begin{displaymath}
 h_{\mathrm{alg}}(\phi):=\sup_{\mathcal{F}\subseteq\Gamma \ \textrm{finite}}\limsup\limits_{n\to\infty}\frac{1}{n}\log\lvert\mathcal{F}_n\rvert,
\end{displaymath}
where $\mathcal F_n : = \mathcal{F}\phi(\mathcal{F}) \cdots \phi^{n-1} (\mathcal F).$
\end{defn}

We say a subspace $A$ of $C^{\ast}_{r}(\mathcal{G},\omega)$ contains $\Gamma$ if it contains the subset $\{\delta_{t}1:t\in\Gamma\}$. Recall that the set $\{[\delta_{t}1]:t\in\Gamma\}$ is an orthonormal set in the GNS space $L^{2}(\tau)$. Then applying Voiculescu's lemma, a straightforward adaptation of the argument of the proof of \cite{chattopadhyay2026metric}*{Theorem 3.8} gives us the following theorem: 

\begin{thm}
    \label{entropy_lowerbound}
    Let $\Phi$ be an automorphism of a twisted transformation groupoid $C^{\ast}$-algebra $C^{\ast}_{r}(\mathcal{G},\omega)$ induced by an automorphism $\varphi$ of $\Gamma$ and a homeomorphism $\psi$ of $X$; $A$ be a subspace of $C^{\ast}_{r}(\mathcal{G},\omega)$ that contains $\Gamma$. Then
    \begin{displaymath}
        h_{\mathrm{alg}}(\varphi)\leq \mathrm{Ent}_{A}(\Phi).
    \end{displaymath}
\end{thm}
Using the fact that the algebraic entropy of the trivial automorphism of a finitely generated discrete group with exponential growth is $+\infty$ (\cite{Dikranjan2013TopologicalEA}*{Proposition 5.3.13.}), we get the following corollary:
\begin{cor}
\label{infinityentropy}
   Let $\mathcal{G}=\Gamma\rtimes X$ and $\omega$ be some $2$-cocycle on $\mathcal{G}$ where $\Gamma$ has exponential growth. For any subspace $A$ of $C^{\ast}_{r}(\mathcal{G},\omega)$ containing $\Gamma$, 
    \begin{displaymath}\mathrm{Ent}_{A}(\mathrm{id})=+\infty\end{displaymath}
\end{cor}
\begin{thm}
    \label{infinity_Mdim}
    Let $\Gamma$ be a finitely generated discrete group of exponential growth acting on a compact metric space $X$; $\omega$ be a $2$-cocycle on the transformation groupoid $\mathcal{G}=\Gamma\rtimes X$; $L$ be a \clipnorm\ on $C^{\ast}_{r}(\mathcal{G},\omega)$ satisfying the Leibnitz property such that $\mathrm{Dom}(L)$ contains $\Gamma$. Then
    \begin{displaymath}
        \mathrm{Mdim}_{L}(C^{\ast}_{r}(\mathcal{G},\omega))=+\infty.
    \end{displaymath}
\end{thm}
\begin{proof}
   If possible let $(C^{\ast}_{r}(\mathcal{G},\omega),L)$ have finite metric dimension. Consider the trivial automorphism of $C^{\ast}_{r}(\mathcal{G},\omega)$. The trivial automorphism is canonically isometric in the sense of \cite{Kerr} for any CQMS structure. Therefore, by \cite{Kerr}*{Corollary 5.5}, $\mathrm{Ent}_{L}(\mathrm{id})=0$. But as $\mathrm{Dom}(L)$ contains $\Gamma$, by Corollary \ref{infinityentropy}, $\mathrm{Ent}_{\mathrm{Dom}(L)}(\mathrm{id})=+\infty$, a contradiction. Therefore, $\mathrm{Mdim}_{L}(C^{\ast}_{r}(\mathcal{G},\omega))=+\infty$.   
\end{proof}
Note that the above theorem does not apply to stratified \clipnorm \ as it does not satisfy the Leibnitz property.
\begin{rem}
    Any \clipnorm\ $L$ on $C^{\ast}_{r}(\Gamma_{1} \rtimes X_{1},\omega_1)$ coming from a spectral triple on natural dense $\ast$-subalgebra $C_{c}(\Gamma, C(X))$ must satisfy the hypothesis of Theorem \ref{infinity_Mdim}. So if $\Gamma$ has exponential growth, one cannot obtain a CQMS structure on $C^{\ast}_{r}(\Gamma\rtimes X,\omega)$ coming from a spectral triple that can ensure finite metric dimension. Also, for the same reason, one can not ensure finite metric dimension for the CQMS structure coming from ergodic action of a compact group as in \cite{Rieffel-1998-Metrics-on-state-from-action-of-cmpt-gp}.
\end{rem}
\subsection{Application: rigidity of CQMS structure} Recall the notion of bi-Lipschitz equivalence of two CQMS structures from Definition \ref{bi-lipschitz_equivalence}. Let $X_1,X_2$ be two compact metric spaces where $X_1$ has finite Kolmogorov dimension. Let $\Gamma_{1}$ be a finitely generated discrete group acting on $X_{1}$ and $\Gamma_{2}$ be a finitely generated discrete group of exponential growth acting on $X_{2}$. Let $\omega_{1},\omega_{2}$ be two normalized $2$-cocycles on $\Gamma_1\rtimes X_1$ and $\Gamma_2\rtimes X_{2}$ respectively. Furthermore, assume that $\Gamma_{1}\rtimes X_{1}$ admits a length function $\ell_{1}$ such that it has polynomial growth of degree say $r$ with respect to $\ell_1$. Then one has a CQMS structure on $C^{\ast}_{r}(\Gamma_{1}\rtimes X_{1},\omega_1)$ for the static stratification and recall the corresponding \clipnorm\ $L^{\mathcal{K}_{S},k_1}_{\ell_1}$ for $k_1>\frac{r}{2}$ from subsection \ref{static_stratification}. Furthermore, assume that $\ell_1$ satisfies Condition \ref{Lip-length}. Similarly assuming twisted rapid decay property on $\Gamma_{2}\rtimes X_{2}$ with respect to a length function $\ell_2$, one obtains a CQMS structure on $C^{\ast}_{r}(\Gamma_{2} \rtimes X_{2},\omega_2)$ coming from the \clipnorm\ $L^{\mathcal{K}_{S},k_{2}}_{\ell_2}$ for $k_{2}$ greater than the rapid decay exponent. With this set up, one has the following theorem:
\begin{thm}
    The CQMS's $(C^{\ast}_{r}(\Gamma_{1}\rtimes X_{1},\omega_1), L^{\mathcal{K}.k_1}_{\ell_1})$ and $(C^{\ast}_{r}(\Gamma_{2}\rtimes X_{2},\omega_2),L^{\mathcal{K}.k_2}_{\ell_2})$ are not bi-Lipschitz equivalent.
\end{thm}
\begin{proof}
 By assumption, the Kolmogorov dimension of $X_1$ is finite. Therefore, $\mathrm{Mdim}_{L^{\mathcal{K}_{S},k_1}_{\ell_1}} (C_r^{\ast}(\Gamma_1\rtimes X_{1}, \omega_1))$ is finite  by Theorem \ref{general cocycle bound}. But by Theorem \ref{stratification_infinity}, $\mathrm{Mdim}_{L^{\mathcal{K}_{S},k_2}_{\ell_2}} (C_r^{\ast}(\Gamma_2\rtimes X_{2}, \omega_2))$ is $+\infty$. Hence, the result follows from Theorem \ref{bi-lipschitz_equivalence}.  
\end{proof}
As observed earlier, taking $X_{1}$ and $X_{2}$ to be singleton point $\{\ast\}$ with $\Gamma_{1},\Gamma_2$ acting trivially, we recover the corresponding result for group $C^{\ast}$-algebras and extend for their cocycle twist.
\begin{cor}
    Let $\Gamma_{1}$ be a discrete group of polynomial growth of growth exponent $r$ with respect to a length function $\ell_{1}$; $\Gamma_{2}$ be a discrete group of exponential growth for some length function $\ell_{2}$ with the rapid decay property with decay exponent $p;$ $\omega_{1},\omega_{2}$ be normalized $2$-cocycles on $\Gamma_{1},\Gamma_{2}$ respectively. Then for any $k_{1}>\frac{r}{2}$ and $k_{2} > p,$ the CQMS's $(C^{\ast}_{r}(\Gamma_{1},\omega_{1}),L_{\ell_1}^{k_1})$ and $(C^{\ast}_{r}(\Gamma_{2},\omega_{2}),L_{\ell_2}^{k_2})$ are not bi-Lipschitz equivalent.   
\end{cor}
In fact, thanks to Theorem \ref{infinity_Mdim}, the dichotomy extends beyond the specific CQMS structure coming from stratified \clipnorm\ on the exponential side. 
\begin{thm}
    The CQMS $(C^{\ast}_{r}(\Gamma_{1}\rtimes X_{1},\omega_1), L^{\mathcal{K}.k_1}_{\ell_1})$ cannot be bi-Lipschitz equivalent to any CQMS structure on $C^{\ast}_{r}(\Gamma_{2}\rtimes X_{2},\omega_2)$ given by any \clipnorm\ $L$ such that $L$ has the Leibnitz property and $\mathrm{Dom}(L)$ contains $\Gamma_2$ in the sense of Theorem \ref{infinity_Mdim}.
\end{thm}
Again taking $X_{1}$ and $X_{2}$ to be singleton point $\{\ast\}$ with $\Gamma_{1},\Gamma_2$ acting trivially, we get the corresponding statement for cocycle twists of group $C^{\ast}$-algebras.
\begin{cor}
Let $\Gamma_1$ be a finitely generated discrete group of polynomial growth with growth exponent $r$ with length function $\ell_1$; $\Gamma_2$ be a group of exponential growth; $\omega_{1},\omega_2$ be two $2$-cocycles on $\Gamma_1,\Gamma_2$ respectively. Then for any $k_1>\frac{r}{2}$, the CQMS $(C^{\ast}_{r}(\Gamma_1,\omega_1),L^{k_1}_{\ell_1})$ can not be bi-Lipschitz equivalent to any CQMS structure on $C^{\ast}_{r}(\Gamma_1,\omega_1)$ given by any \clipnorm\ $L$ which satisfies the Leibnitz property and $\mathrm{Dom}(L)$ contains $\Gamma_2.$   
\end{cor}

\vspace{2mm}

\bibliographystyle{amsplain}
\bibliography{References}

@article {Kerr,
    author = {Kerr, D.},
     title = {Dimension and {D}ynamical {E}ntropy for {M}etrized {$C^{\ast}$}-algebras},
   journal = {Comm. Math. Phys.},
  fjournal = {Documenta Mathematica},
    volume = {32},
      year = {2003},
     pages = {501-534},
      issn = {1431-0635,1431-0643},
}

@book{brownc,
  title={{$C^{\ast}$}-algebras and {F}inite-dimensional {A}pproximations},
  author={Brown, N.P. and Ozawa, N.},
  isbn={9780821872505},
  series={Graduate {S}tudies in {M}athematics},
  url={https://books.google.co.in/books?id=F_kjj0teG2IC},
  year={2008},
  publisher={American Mathematical Soc.}
}

@book{williams2019tool,
  title={A {T}ool {K}it for {G}roupoid {$C^{\ast}$} -{A}lgebras},
  author={Williams, D.P.},
  isbn={9781470451332},
  lccn={2019016171},
  series={Mathematical Surveys and Monographs},
  url={https://books.google.co.in/books?id=5q6xDwAAQBAJ},
  year={2019},
  publisher={American Mathematical Society}
}

@book{sims2020operator,
  title={Operator Algebras and Dynamics: Groupoids, Crossed Products, and Rokhlin Dimension},
  author={Sims, A. and Szab{\'o}, G. and Williams, D.},
  isbn={9783030397135},
  series={Advanced Courses in Mathematics - CRM Barcelona},
  url={https://books.google.co.in/books?id=BdfsDwAAQBAJ},
  year={2020},
  publisher={Springer International Publishing}
}

@article {Rieffel-1998-Metrics-on-state-from-action-of-cmpt-gp,
    AUTHOR = {Rieffel, M. A.},
     TITLE = {Metrics on states from actions of compact groups},
   JOURNAL = {Doc. Math.},
  FJOURNAL = {Documenta Mathematica},
    VOLUME = {3},
      YEAR = {1998},
     PAGES = {215--229},
      ISSN = {1431-0635,1431-0643},
   MRCLASS = {46L87 (46L55 58B30 60B10)},
  MRNUMBER = {1647515},
MRREVIEWER = {A.\ I.\ Danilenko},
}

@article {Long-Wu-2017-Twisted-group-C-alg-as-CQMS,
    AUTHOR = {Long, B. and Wu, W.},
     TITLE = {Twisted {G}roup {$C^{\ast}$}-{A}lgebras as {C}ompact {Q}uantum {M}etric {S}paces},
   JOURNAL = {Results Math.},
  FJOURNAL = {Results in Mathematics},
    VOLUME = {71},
      YEAR = {2017},
    NUMBER = {3-4},
     PAGES = {911--931},
      ISSN = {1422-6383,1420-9012},
   MRCLASS = {46L85 (20J06 22D15 58B34)},
  MRNUMBER = {3648451},
       DOI = {10.1007/s00025-016-0562-7},
       URL = {https://doi.org/10.1007/s00025-016-0562-7},
}

@article {Rieffel-2002-Group-C-alg-as-compact-quan-metr-sp,
    AUTHOR = {Rieffel, M. A.},
     TITLE = {Group {$C^{\ast}$}-algebras as compact quantum metric spaces},
   JOURNAL = {Doc. Math.},
  FJOURNAL = {Documenta Mathematica},
    VOLUME = {7},
      YEAR = {2002},
     PAGES = {605--651},
      ISSN = {1431-0635,1431-0643},
   MRCLASS = {22D25 (20F65 20F69 46L55 46L87 58J42)},
  MRNUMBER = {2015055},
MRREVIEWER = {Alain\ Valette},
}

@article{Christensen, author={Antonescu, C. and Christensen, E.},
 title={Metrics on group {$C^{\ast}$}-algebras and a non-commutative {A}rzel\`a-{A}scoli theorem},
   Journal={J. Funct. Anal.},
   volume={214}, 
   year={2004},
   Issue={2},
   pages={247-259}
   }

@article{Connes1,
   author={Connes, A.},
   title={Compact metric spaces, {F}redholm modules and hyperfiniteness},
   journal={Ergod. Th. Dynam. Sys.},
   volume={9},
   date={1989},
   pages={207-220},
   year={1989}
   }

@book{delaharpe,
   author={de la Harpe, P.},
   title={Topics in {G}eometric {G}roup {T}heory},
   publisher={The University of Chicago Press},
   year={2000}
   }

@article{Rieffel2,
   author={Ozawa, N. and Rieffel, M.A.},
   title={Hyperbolic group {$C^{\ast}$}-algebras and free product {$C^{\ast}$}-algebras as compact quantum metric space},
   Journal={Canad. J. Math.},
   volume={57(5)},
   year={2005},
   pages={1056-1079}
   }

@article{Rieffel1,
   author={Rieffel, M.A.},
  title={Metrics on state spaces},
   Journal={Doc. Math.},
   volume={4}, 
   year={1999},
   pages={559-600}
   }

@article{Voiculescu,
   author={Voiculescu, D.},
   title={Dynamical approximation entropies and topological entropy in operator algebras},
   journal={Comm. Math. Phys.},
   volume={170},
   year={1995},
   pages={249-281}
   }

@article{nica2010degree,
  title={On the degree of rapid decay},
  author={Nica, B.},
  journal={Proc. Amer. Math. Soc.},
  volume={138},
  number={7},
  pages={2341--2347},
  year={2010}
}

@article{Dikranjan2013TopologicalEA,
  title={Topological Entropy and Algebraic Entropy for group endomorphisms},
  author={Dikranjan, D. and Bruno, A. G.},
  journal={arXiv: General Topology},
  year={2013},
  url={https://api.semanticscholar.org/CorpusID:118325710}
}

@article {MR943303,
    AUTHOR = {Jolissaint, P.},
     TITLE = {Rapidly decreasing functions in reduced {$C^*$}-algebras of
              groups},
   JOURNAL = {Trans. Amer. Math. Soc.},
  FJOURNAL = {Transactions of the American Mathematical Society},
    VOLUME = {317},
      YEAR = {1990},
    NUMBER = {1},
     PAGES = {167--196},
      ISSN = {0002-9947,1088-6850},
   MRCLASS = {22D25 (43A15 46L99)},
  MRNUMBER = {943303},
MRREVIEWER = {A.\ Derighetti},
       DOI = {10.2307/2001458},
       URL = {https://doi.org/10.2307/2001458},
}

@article{austad2026quantum,
  title={Quantum metrics from length functions on \'etale groupoids},
  author={Austad, A.},
  journal={arXiv preprint arXiv:2602.20032},
  year={2026}
}

@article {Weygandt-2024-Rapid-decay-for-etale-gpd,
    AUTHOR = {Weygandt, A.},
     TITLE = {Rapid decay for principal \'etale groupoids},
   JOURNAL = {New York J. Math.},
  FJOURNAL = {New York Journal of Mathematics},
    VOLUME = {30},
      YEAR = {2024},
     PAGES = {956--978},
      ISSN = {1076-9803},
   MRCLASS = {22A22 (46L05)},
  MRNUMBER = {4773307},
}

@article{chattopadhyay2026metric,
  title={Metric dimension and product entropy of group {$C^{\ast}$}-algebras},
  author={Chattopadhyay, A. and Joardar, S.},
  journal={arXiv preprint arXiv:2603.13936},
  year={2026}
}

@article{ARMSTRONG2022109551,
title = {A uniqueness theorem for twisted groupoid {$C^{\ast}$}-algebras},
journal = {J. Funct. Anal.},
volume = {283},
number = {6},
pages = {109551},
year = {2022},
issn = {0022-1236},
doi = {https://doi.org/10.1016/j.jfa.2022.109551},
url = {https://www.sciencedirect.com/science/article/pii/S0022123622001719},
author = {Armstrong, B.},
}

@article{renault2008cartan,
  title={Cartan Subalgebras in {$C^*$}-Algebras},
  author={Renault, J.},
  journal={rish Math. Soc. Bull.},
  volume={61},
  pages={29--63},
  year={2008},
  url={https://tcd.ie},
}

@article{austad2024polynomial,
  author    = {Austad, A. and Ortega, E. and Palmstr{\o}m, M.},
  title     = {Polynomial growth and property {$\mathrm {RD}_p$} for {\'e}tale groupoids with applications to {K}-theory},
  journal   = {J. Noncommut. Geom.},
  volume    = {18},
  number    = {2},
  pages     = {727--757},
  year      = {2024},
  publisher = {EMS Press},
  doi       = {10.4171/JNCG/549},
}

@article{joardar2025metrics,
  title={Metrics on {$C^{\ast}$}-algebras of \'Etale groupoids from length functions},
  author={Chattopadhyay, A. and Hossain, Md A. and Joardar, S.},
  journal={arXiv preprint arXiv:2504.13530},
  year={2025},
  pages={13},
}

@article{emerson2018k,
  title={K-homological finiteness and hyperbolic groups},
  author={Emerson, H. and Nica, B.},
  journal={J. Reine Angew. Math.},
  volume={2018},
  number={745},
  pages={189--229},
  year={2018},
  publisher={De Gruyter}
}

@article{latremoliere2016quantum,
  title={The quantum {G}romov-{H}ausdorff propinquity},
  author={Latr{\'e}moli{\`e}re, F.},
  journal={Trans. Amer. Math. Soc.},
  volume={368},
  number={1},
  pages={365--411},
  year={2016}
}

@article{latremoliere2015quantum,
  title={Quantum metric spaces and the {G}romov-{H}ausdorff propinquity},
  author={Latr{\'e}moli{\`e}re, F.},
  journal={arXiv preprint arXiv:1506.04341},
  year={2015}
}

@article{latremoliere2022gromov,
  title={The {G}romov-{H}ausdorff propinquity for metric spectral triples},
  author={Latr{\'e}moli{\`e}re, F.},
  journal={Adv. Math.},
  volume={404},
  pages={108393},
  year={2022},
  publisher={Elsevier}
}

@article{latremoliere2013quantum,
  title={The Quantum {G}romov-{H}ausdorff Propinquity},
  author={Latr{\'e}moli{\`e}re, F.},
  journal={arXiv preprint arXiv:1302.4058},
  year={2013}
}

@article{latremoliere2020convergence,
  title={Convergence of {H}eisenberg modules over quantum 2-tori for the modular {G}romov-{H}ausdorff propinquity},
  author={Latr{\'e}moli{\`e}re, F.},
  journal={J. Operator Theory},
  volume={84},
  number={1},
  pages={211--237},
  year={2020},
  publisher={JSTOR}
}

@article{buss2026rapid,
  title={Rapid decay and localizability for Fell bundles over {\'e}tale Groupoids},
  author={Buss, A. and Karmakar, P.},
  journal={arXiv preprint arXiv:2604.23907},
  year={2026}
}

@inproceedings{aguilar2023domains,
  title={Domains of quantum metrics on {AF} algebras},
  author={Aguilar, K. and Hjelmborg, K.B. and Latr{\'e}moli{\`e}re, F.},
  booktitle={International Workshop on Operator Theory and its Applications},
  pages={1--14},
  year={2023},
  organization={Springer}
}

@book{renault1980groupoid,
  title={A groupoid approach to {$C^{\ast}$}-algebras},
  author={Renault, J.},
  volume={793},
  year={1980},
  publisher={Springer-Verlag}
}

@article{bowen1971entropy, 
 title={Entropy for group endomorphisms and homogeneous spaces}, 
 author={Bowen, R.}, 
 journal={Trans. Amer. Math. Soc.}, 
 volume={153}, 
 pages={401--414}, 
 year={1971}, 
 publisher={American Mathematical Society}, 
 doi={10.1090/S0002-9947-1971-0274707-X},
 }

@article{Chatterji2016,
  author    = {Chatterji, I.},
  title     = {Introduction to the Rapid Decay property},
  journal   = {arXiv preprint arXiv:1604.06387},
  year      = {2016},
  url       = {https://arxiv.org/abs/1604.06387}
}

@article{ChatterjiRuane2005,
  author    = {Chatterji, I. and Ruane, K.},
  title     = {Some geometric groups with rapid decay},
  journal   = {Geom. Funct. Anal.},
  volume    = {15},
  number    = {2},
  pages     = {311--339},
  year      = {2005},
  publisher = {Springer}
}

@article {MR2504529,
    AUTHOR = {Li, H.},
     TITLE = {Compact quantum metric spaces and ergodic actions of compact
              quantum groups},
   JOURNAL = {J. Funct. Anal.},
  FJOURNAL = {Journal of Functional Analysis},
    VOLUME = {256},
      YEAR = {2009},
    NUMBER = {10},
     PAGES = {3368--3408},
      ISSN = {0022-1236,1096-0783},
   MRCLASS = {46L55 (20G42)},
  MRNUMBER = {2504529},
MRREVIEWER = {Kenny\ De Commer},
       DOI = {10.1016/j.jfa.2008.09.009},
       URL = {https://doi.org/10.1016/j.jfa.2008.09.009},
}

@article {MR124720,
    AUTHOR = {Kolmogorov, A. N. and Tihomirov, V. M.},
     TITLE = {{$\varepsilon $}-entropy and {$\varepsilon $}-capacity of sets
              in functional space},
   JOURNAL = {Amer. Math. Soc. Transl. (2)},
  FJOURNAL = {Amer. Math. Soc. Transl. (2)},
    VOLUME = {17},
      YEAR = {1961},
     PAGES = {277--364},
   MRCLASS = {46.30 (41.00)},
  MRNUMBER = {124720},
}

@book {MR1187754,
    AUTHOR = {Makarov, B. M. and Goluzina, M. G. and Lodkin, A. A. and
              Podkorytov, A. N.},
     TITLE = {Selected problems in real analysis},
    SERIES = {Transl. Math. Monogr.},
    VOLUME = {107},
      NOTE = {Translated from the Russian by H. H. McFaden},
 PUBLISHER = {American Mathematical Society, Providence, RI},
      YEAR = {1992},
     PAGES = {x+370},
      ISBN = {0-8218-4559-4},
   MRCLASS = {26-01 (00A07)},
  MRNUMBER = {1187754},
MRREVIEWER = {R.\ G.\ Bartle},
       DOI = {10.1090/mmono/107},
       URL = {https://doi.org/10.1090/mmono/107},
}

\end{document}